\documentclass[reqno, 12pt, final ]{amsart}

\usepackage{amssymb,amsfonts,showkeys,graphicx,amsmath,amsthm,amstext,latexsym,amsfonts,mathrsfs}
\usepackage[latin1]{inputenc}
\usepackage{pstricks,pst-plot}
 \theoremstyle{plain}
 \newtheorem{theorem}{Theorem}[section]
 \newtheorem{definition}[theorem]{Definition}
 \newtheorem{lemma}[theorem]{Lemma}
 
 \numberwithin{equation}{section} %% Comment out for sequentially-numbered
 \numberwithin{figure}{section} %% Comment out for sequentially-numbered
 \newtheorem{prop}[theorem]{Proposition}
 \theoremstyle{remark}
 \newtheorem{remark}[theorem]{Remark}
 
 \newtheorem{example}[theorem]{Example}
\newtheorem*{acknowledgment*}{Acknowledgment}
\newcommand{\mean}[1]{\, -\hskip-1.08em\int_{#1}}
\def\Refo{{\rm Ref}}
\def\GRefo{{\rm GRef}}
\def\lip{{\rm Lip}}
\def\eps{\varepsilon}
\def\R{{\mathbb R}} 
\def\N{{\mathbb N}}

\def\M{{\mathcal M}} 
\def\C{{\mathcal C}} 
\def\H{{\mathcal H}} 

\def\L{{\mathcal L}}

\def\P{{\mathcal P}}

\def\res{\mathop{\hbox{\vrule height 7pt width.5pt depth 0pt\vrule height .5pt width 6pt depth
0pt}}\nolimits}
\title[Elastic reformations]{A Metric Approach to  Elastic reformations }
\author{Luca Granieri, Francesco Maddalena}
\address{Dipartimento di Matematica Politecnico di Bari, via Orabona 4, 70125 Bari, Italy}
\email{granieriluca@libero.it, l.granieri@poliba.it,  f.maddalena@poliba.it}
\date{}

\begin{document}
\begin{abstract}
We study a variational framework to compare shapes, modeled as Radon measures on $\R^N$, in order to quantify how they differ from isometric copies. To this purpose we discuss some notions of \emph{weak deformations} termed \emph{reformations} as well as integral functionals having some kind of \emph{isometries} as minimizers. The approach pursued is 
based on the notion of pointwise Lipschitz constant leading  to  a space metric  framework. In particular, to compare general shapes, we study this \emph{reformation} problem by using the notion of transport plan and  Wasserstein distances as in optimal mass transportation theory.

\end{abstract}

\maketitle
%\begin{center}\large\sffamily DRAFT\end{center}
\tableofcontents
\begin{flushleft}
\end{flushleft}

\section*{Introduction}
One of the main goal in shape analysis relies in detecting and quantifying differences between shapes. The interest for such studies   concerns a wide range of applications, especially those within the computer vision community, in particular in pattern recognition,
image segmentation, and computation anatomy (see \cite{shape-book, trouve, trouve2}).
In recent years many authors have focused their attention on the notions of 
\textit{shape space} and  \textit{shape metric} to the aim of  establishing a general framework in which  the analysis of shapes crucially depends on their invariance with respect to suitable geometric transformations (see \cite{FJSY, trouve, sur-cur, SSJD}). 
A natural suggestion in this direction comes from continuum mechanics since the variational theory of elasticity  can be used to compare the initial and final shape of  a deformable  material body, i.e. to establish  how  the two shapes differ from an isometry of the euclidean space. Therefore some authors begin to study 
the possible links between elastic energies and distances in shape spaces (see  \cite{FJSY, WBRS, wolan}).\\
On the other side, by arguing  from a mechanical perspective, we know that a large class of physical manifestations (fractures, fragmentations,  material instabilities)  require  more general kinematical  tools
than those available in the context of  Sobolev maps, hence it seems  reasonable to exploit  a more general mathematical framework to obtain more accurate  descriptions of more complex physical problems.\\
In this paper we model (material) shapes as Radon measures on subsets  $X,Y\subset\R^N$ and study  a variational model to the aim of quantifying  how a target shape $\nu$ on $Y$ differs from an isometric copy of $\mu$ on $X$.
To this purpose  we  scrutiny  some  notions of \emph{weak} deformations, which we denote by the term \textit{reformations},  as well  as energy like (or cost) functionals having some kind of  \emph{isometries} between $\mu$ and $\nu$ as minimizers. \\
In the first part of the paper (Sections 1,2,3)  we study the variational problem of \textit{reformation} of two shapes $\mu$ and $\nu$ through  functions called \textit{reformation maps},  while in the second part (Sections 4,5,6) we relax the problem  by  considering  a formulation in terms of
 \textit{transport plans}  which leads to a  variational framework   as in optimal  transport theory.\\
 For reader convenience we have added an appendix containing some basic tools from analysis in metric spaces.
%%%%%%%%%%%%%%%%%%%%%%%%%%%%%%%%%%%%%%%%%%%%%%%%%%%%%%%%%%%%%%%%%%%%%%%%%%%%%%%%%%%%%%%%%%%%%%%%%%%%%%%%%%%%%%%%%%%%%%%%%%%%%
\section*{Description of the variational model and main results}
%The general question addressed in this paper is to provide some variational tools able 
To \emph{quantify} how two shapes $X,Y\subset\R^N$ are close to be \emph{isometric}, an usual way  relies in considering $Y=u(X)$ for maps $u$  belonging to a suitable class of admissible maps. The two shapes are   \emph{isometric} if there exists   $u:X\rightarrow Y$ such that $u(X)=Y$ and   
$$|u(x)-u(y)|=|x-y|, \:\:\:\forall x,y\in X.$$
Equivalently, the last condition means that the map $u$ has   bi-Lipschitz constant $L=1$. Let us recall that a map $u: x \rightarrow Y $ is said to be bi-Lipschitz with constant $L$ if 
$$ \frac{1}{L}|x-y|\leq |u(x)-u(y)|\leq L|x-y|,\:\:\:\forall x,y\in X.$$ 
Therefore, the two shapes $X,Y$ could be considered close to be \emph{isometric} as the bi-Lipschitz constant $L$ is close to one, so assuming the bi-Lipschitz constant as a quantifier of the closeness to the isometry.
 This approach has the disadvantage to involve  a global condition. For instance,  the shapes in Figure $0.1$ %\ref{bending-fig}
 looks very close to be isometric but the bi-Lipschitz constant is quite large and far from $L=1$, whatever the size of the bending part. 
 %%%%%%%%%%%%%%%%%%%%%%%%%%%%%%%%%%%%%%%%%%%%%%%%%%%%%%%%%%%%%%%%%%%%%%%%%
 \begin{figure}[htbp]\label{bending-fig}\vspace{1cm}\begin{pspicture}(-4.5, -1.5)(12.5, 3.5)\pspolygon(-1,-1)(-1,4)(1,4)(1,-1)\pspolygon(4,-1)(4,3)(7,3)(7,1)(6,1)(6,-1)\psline[linestyle=dashed](6,3)(6,1)\uput*[0](6.1,2){$\frac{1}{n}$}\end{pspicture}\caption{Bending a rectangle.}\end{figure}
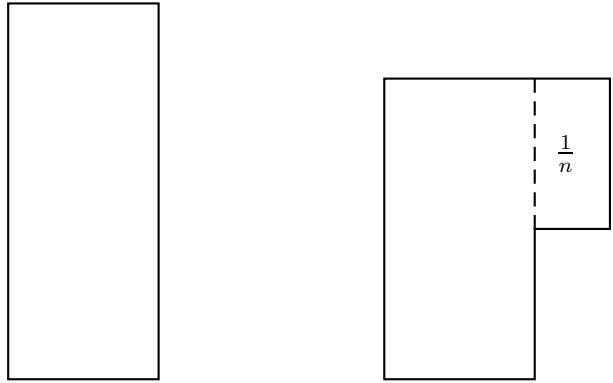
 %%%%%%%%%%%%%%%%%%%%%%%%%%%%%%%%%%%%%%%%%%%%%%%%%%%%%%%%%%%%%%%%%%%%%%%%%%%%%%%%%%%%%%%%%%%
   To avoid this difficulty some \emph{localization} procedure is needed. 
 This can be done by an analytical approach.\\
  An isometry $u$ is of course an affine map $u(x)=Ax+b$ and  $\nabla u=A$ is an orthogonal matrix. Actually, under some regularity assumptions, by Liouville Rigidity Theorem the orthogonality of the 
 Jacobian matrix characterizes  the isometric maps (see also 
 Theorem~\ref{rigidity}). Hence, a reasonable way to quantify how two shapes are isometric is that of measuring how  $\nabla u$ is close to be an orthogonal matrix. 
This program can be carried on by selecting a function $W$ reaching its minimal value at the  orthogonal matrices. Then, by Liouville  Rigidity Theorem, it follows that the isometries characterize  the minimal possible value of  the functional $I(u)=\int_\Omega W(\nabla u)\ dx$.
This approach is pursued in \cite{wolan} for smooth $2$-dimensional domains where the admissible  maps  are  
incompressible  diffeomorphisms, i.e. $u$ such that $\det \left (\nabla u \right )=1$.\\ 
 In order to characterize the isometries,  a polyconvex function $W$ having minimal value at orthogonal matrices is selected.  Therefore, to quantify how two domains $\Omega_1, \Omega _2\subset\R^N$ are close to be isometric one considers the variational problem
$$\min \left \{ \int_{\Omega _1} W(\nabla u)\ dx\:\vert\: u(\Omega_1)=\Omega _2,\: u\in \mathcal D \right \}, $$ 
where $\mathcal D$ denotes the set of 
incompressible diffeomorphisms.
This approach has of course many restrictions. For instance, to compare a connected domain to a disconnected one, or for non-smooth domains,  many regularity questions arise. \\
  %This discussion is reminiscent of the variational principles 
  In continuum mechanics, 
  %indeed  in the setting of nonlinear elasticity, in the undeformed state the body occupies a bounded open set $\Omega \subset \R ^N$,  and 
  one usually looks for minimizers $u :\overline \Omega \rightarrow \R^N$ of the stored energy $I(u)=\int_\Omega W(\nabla u)\ dx$ in an admissible class  of deformations usually  consisting  of Sobolev functions which are locally orientation preserving, i.e. ${\rm det} \nabla u(x) >0$ for a.e. $x\in \Omega$. 
%For several reasons (see also \cite{g-m-s}) this approach could be not appropriate. For instance, the constraint imposed on the Jacobian determinant sign makes not clear if a minimizer $u$ satisfies the corresponding Euler-Lagrange equilibrium equation$${\rm div} \frac{\partial W}{\partial F} (\nabla u)=0 \, \mbox{ in } \Omega,$$ since the variations of $u$ may violate such a constraint  in a subset of positive measure.\\ 

A main goal of our approach relies in exploiting possible extensions of  the variational scheme of elasticity %(see also \cite{g-m-s} for an analysis of this topic)
 in order to compare more general shapes as those in Figure $4.1 $, %\ref{piano}, 
also allowing \emph{fragmentations}.   However, a purely measurable setting does not work to compare shapes as shown in Example \ref{ex} and, on the other hand, to compare a more extended class of shapes we have to reduce  regularity requirements. So, a useful compromise relies in  working on a general metric framework.\\
We remark that an  approach like the one followed in  \cite{wolan}  cannot be pursued in a metric framework,  indeed the mapping 
 $A\mapsto\varphi(\|A\|)$ is polyconvex only if $\varphi$ is a positive convex and strictly increasing function (see for instance \cite{dac}),  therefore the minimal value cannot be reached at orthogonal matrices $A$, since they have $\|A\|=1$.\\
 We denote by $\P(X)$ the space of probability measures on the metric space  $X$.
 Assume the material shapes are given by probability measures 
 $\mu\in \P(X)$ and $\nu\in \P(Y)$,  (to fix ideas consider $\mu =\L^N\res X$, 
 $\nu =\L^N\res Y$).\\
 In this paper we assume the \textit{pointwise Lipschitz constant} $\lip(u)(x)$ (see Definition~\ref{plc}) as a  local descriptor to measure how an admissible map $u$ is an \emph{expansion} or a \emph{contraction}. Note that   the  \textit{pointwise Lipschitz constant} $\lip(u)(x)$  coincides with the operator norm 
$\|\nabla u (x)\|$ whenever $u$ is a  differentiable map.  
Hence the local expansion and contraction due to the map $u$ at any point $x$ are respectively represented by the functions $e_u(x)$ and $c_u(x)$ (see Definition~\ref{c-e}) depending on $\lip(u)(x)$.\\
 Let $H(x),K(x)>0$ be given. We require the admissible maps $u:X\rightarrow Y$ satisfy the conditions  
 \begin{equation}
 \label{push}
 u_\# \mu=\nu,
 \end{equation}
 \begin{equation}
e_u\leq K,\:\: c_u\leq H .
 \end{equation}
 %for (locally uniformly) given $H,K>0$.
 These maps will be called  \emph{reformation maps}, and the set of such functions will be   denoted by $\Refo(\mu;\nu)^{H,K}$ (see Definition \ref{def-ref}). We consider the local reformation  energy  $r_u=e_u+c_u$ and the total reformation energy  $\mathscr{R} (u)=\int _X r_u(x) \ d\mu$, so   in Theorem \ref{min1} we show that the variational problem 
\begin{equation}
\min\{ \mathscr{R}(u) \;\vert\; u\in \Refo(\mu ;\nu)^{H,K} \}
 \end{equation}
  admits solutions whenever $\Refo(\mu ;\nu)^{H,K}\neq \emptyset$. 
Therefore, to quantify how the two measures $\mu, \nu$ are isometric we look to the \emph{elastic reformation energy}
 between $\mu$ and $\nu$ defined by 
\begin{equation}\label{ere}\mathcal E(\mu,\nu):= \hbox{\rm inf} \{ \mathscr{R}(u) \:\vert\:  u\in \Refo(\mu;\nu)^{H,K}\}.
\end{equation}
In Theorem \ref{main} we show that the  value of \eqref{ere} is attained and equal to $2$ if and only if the two shapes are isometric.\\
In Section 4 we extend the scenario to deal with  the case of non-existence of  maps satisfying  \eqref{push} and this happens, for instance, when fragmentation occurs. In such a case  the notion of transport plan coming from
the optimal mass transportation or Monge-Kantorovich theory is useful. 
A transport plan between $\mu$ and $\nu$  is a measure  $\gamma \in \P(X\times Y)$, %where $\P$ denotes the space of probability measures,  
having $\mu, \nu$ as marginals, 
namely $(\pi_1)_\# \gamma = \mu, (\pi_2)_\# \gamma = \nu$, where $\pi_{1,2}$ are the projections of $X\times Y$ on its factors. The notion of transport plan  could be considered as a generalization of a transport map, i.e. $u:X\rightarrow Y$, such that $u_\#\mu =\nu$,   and so as a \emph{weak} notion of reformation of $\mu$ into $\nu$. 
We shall refer to such measures  as \textit{reformation plans}.
Actually, to every transport map corresponds the transport plan given by $(I\times u)_\# \mu$. 
%Inspired by the variational methods of elasticity, 
The
 shapes $\mu, \nu$  could be compared by considering the local mass transportation displacement in the target configuration. 
 
 More precisely, by Disintegration Theorem (see \cite[Section 2.5]{AFP}) every transport plan $\gamma \in \P(X\times Y)$ can be written as $\gamma =  f(x)\otimes \mu$, where
 \begin{equation}
   f : X \rightarrow (\P(Y), W) 
   \end{equation}
 is called \textit{disintegration map}   and $\P(Y)$ is equipped with the Wasserstein distance $W$. 
 This point of view leads  to  formulate the reformation problem in terms of disintegration map $f$ and related metric expansion and contraction energies
 (see Definition~\ref{g-c-e}).
 So, in this setting,  
 reformation maps take value in 
the space of  probabilities, endowed with the Wasserstein metric, over
the target domain.
The main advantage of this approach relies in its  generality and in its connections with fertile topics as optimal mass transportation and geometric measure theory.
 However, many interesting open questions arise as the regularity needed on $f$ to capture relevant geometrical and physical properties of the shapes. \\
 In Section 5 we show several examples of shape reformations attainable by disintegration maps but not attainable  by any  regular transport map.\\
 In Section 6 we study the main aspects of the variational problems of reformation in the enlarged context of generalized reformations,   showing in Theorem~\ref{min2} how isometric measures can be characterized by means of the reformation energy. 
 In Theorem~\ref{existence-plan} we prove the existence  of minimizing reformation plans  for a constrained  variational problem.

%%%%%%%%%%%%%%%%%%%%%%%%%%%%%%%%%%%%%%%%%%%%%%%%%%%%%%%%%%%%%%%%%%%%%%%%%%%%%%%%%%%

%~~~~~~~~~~~~~~~~~~~~

 \section{The pointwise Lipschitz constant}
 In this section we introduce the notion of  \textit{pointwise Lipschitz constant} and scrutiny some properties related to this tool since it will play a  crucial role in this  paper.  %Notice that  The pointwise Lipschitz  constant  is a useful tool  in theories concerning  Sobolev spaces in a metric setting (see \cite{cheeger, point-lip, new}). 
\begin{definition} 
\label{plc}
  Let  $(X, d_X)$,  $(Y, d_Y )$  be two  metric spaces and let 
    $f: (X, d_X) \rightarrow ( Y,d_Y )$.  The pointwise Lipschitz  constant
    $\lip(f)(x_0)$ of $f$  at $x_0 \in X$  is defined by

\begin{equation}
   \lip(f)(x_0):= \left\{
   \begin{array}{ll}\displaystyle\limsup_%{\begin{array}{c}
    {x\rightarrow x_0}%    { x \neq x_0}\end{array}}
     \frac{d_Y(f(x),f(x_0))}{d_X(x,x_0)} &  \hbox {if $x_0$ is a non-isolated point}, \\
0  & \hbox {if $x_0$ is an isolated point} .
\end{array}
\right.
\end{equation}
\end{definition}
\noindent It is readily seen  that $\lip(f)$    is a measurable  function. The pointwise Lipschitz constant leads to a global Lipschitz's constant  on convex sets. 

%~~~~~~~~~~~~~~~~~~~~
   \begin{lemma}[Lemma 14.4 of \cite{nonlinear}]
  \label{lip-lemma}
Let $L>0$,   $X\subset \R^N$  a  convex set  and let $f:X\rightarrow (Y,d)$ be a function such that
 $\lip(f)(x)\leq L$  $\forall x\in X$. Then $f$ is $L$-Lipschitz.
\end{lemma}
%\begin{proof}Let $k>L$ and $a,b \in X$ be fixed. We define  the set$$S=\left \{ t\in [0,1] \ : \ d(f(a),f(a+t(b-a))) \leq k t|b-a| \right \}.$$Since$\lip(f)(a)< \mu$ then  $S\neq \emptyset$. Let $\tau=\sup S$. By continuity of $f$, which follows by the boundedness of $\lip(f)(x)$( seealso\cite{point-lip}),  we have that $\tau \in S$.  We claim that $\tau =1$. Otherwise, since  $\lip(f)(a+\tau (b-a))<k$, we find $\bar{t}\in(\tau,1)$  such that $$d(f(a),f(a+\bar{t}(b-a))\leq d(f(a),f(a+\tau(b-a))+d(f(a+\tau(b-a)),f(a+\bar{t}(b-a))\leq$$$$ \tau k |b-a|+(\bar{t}-\tau)|b-a|k =k \bar{t} |b-a|,$$so  $\bar{t}\in S$, in contradiction with the maximality of $\tau$. Letting $k \rightarrow L^+$ we get the thesis. \end{proof}

%~~~~~~~~~~~~~~~~~~~~
 A  result similar to the previous lemma   holds true  for
 \emph{quasi-convex} metric spaces $X$ (see \cite{point-lip}).\\ 
A metric space $(X,d)$ is $C$-quasi-convex if
there exists a constant $C > 0$ such that for each pair of points $x, y \in X$, there exists a curve $\gamma$  connecting $x$ and $y$ with
$l(\gamma)\leq Cd(x, y)$. As one can expect, a metric space is quasi-convex if, and only if, it is bi-Lipschitz homeomorphic to some
length space. For $X$ $C$-convex, the function $f$ of Lemma \ref{lip-lemma} is just $CL$-Lipschitz. 

 The pointwise Lipschitz constant is also related to the notion of metric differential (see \cite{a-t, area-metric, kirc, diff-l}). 
 A function $f:X\subset \R^N \rightarrow (Y,d)$ is said to be metrically differentiable at a point $x_0 \in X$ if there exists a (unique) 
  on $\R^N$, denoted by $MD(f,x_0)$,  such that for every $y, z\in X$ the following formula holds true
 \begin{equation}\label{m-diff} d(f(y),f(z))-MD(f,x_0)(y-z)=o(\|y-x_0\|+\|z-x_0\|). \end{equation}
Let  $U\subset \R^N$ be an open set and let $f:U\rightarrow (Y,d)$ be a Lipschitz function. Hence, for every fixed $p\in Y$ the function 
\begin{equation}
x\mapsto d(f(x),p):U\rightarrow \R_+
\end{equation} 
 is a Lipschitz function and  by Rademacher Theorem it is a.e. differentiable
in $U$.  Moreover (see \cite{ area-metric, kirc}),  it turns out that $f$ is  \emph{metrically} differentiable at almost every point. 
%%%%%%%%%%%%%%%%%%%%%%%%%%%%%%%%%%%%%%%%%%%%%%%%%%%%%%%%%%%%%
%More precisely, it turns out (see \cite{kirc, diff-l}) that  the function $ g:= d(f(\cdot),f(x)):U\rightarrow \R $ is differentiable at $x$, for  a.e. $x\in U$.  Questo non mi sembra coerente. C'ï¿½ qualcosa che non ï¿½ chiaro in \cite{diff-l}. Bisogna rivedere la dimostrazione del lemma \ref{diff-lip}.
%Indeed, Consider a metric differentiability point  $x_0\in U$ for $f$. For a direction $\|v\|=1$ we compute $$ \frac{\partial g}{\partial v}(x_0) =\lim_{t\rightarrow 0}\frac{g(x_0+tv)-g(x_0)}{t}=\lim_{t\rightarrow 0}\frac{d(f(x_0+tv),f(x_0))}{t}.$$ By \eqref{m-diff} it follows that $\frac{\partial g}{\partial v}(x_0)=MD(f,x_0)(v)$.  In terms of the map $g$ \eqref{m-diff} becomes $$ g(x)-g(x_0)- \frac{\partial g}{\partial v}(x_0)=o(\|x-x_0\|).$$ Therefore, $g$ is differentiable at the point $x_0$ whenever $\frac{\partial g}{\partial v}(x_0)=\langle \nabla g(x_0), v \rangle $. For if it suffices to observe that actually $x_0$ is a local minimum point of $g$. 
 
   The following lemma  establishes a  link between the pointwise Lipschitz constant,  the distance function (see \cite{a-k} for dual Banach spaces) and the metric differential (see \eqref{md} below).

\begin{lemma}\label{diff-lip} Let $f:U\subset \R^N \rightarrow (Y,d)$ be a Lipschitz function over a separable metric space $Y$. Then, for a.e. $x_0\in  U$  it results
 \begin{equation}\label{suppo1}
 \lip(f)(x_0)= \sup _{p\in Y}\|\nabla d(f(x_0),p)\|.
\end{equation}
\end{lemma}
\begin{proof}
We assume  $\|\nabla d(f(x_0),p)\|=0$ if the function $x\mapsto d(f(x),p)$ is not 
differentiable at $x_0$. For $p\in Y$ we compute
\begin{equation*} \begin{split}\langle \nabla d(f(x_0),p),v\rangle  &= \lim_{t\rightarrow 0^+}\frac{d(f(x_0+tv),p)-d(f(x_0),p)}{t} \\
&\leq \lim_{t\rightarrow 0^+}\frac{d(f(x_0+tv),f(x_0))}{t} \leq \lip(f)(x_0).\end{split}\end{equation*}
Taking the supremum with respect to $v$ and then respect to $p$, we have $$\sup _{p\in Y}\|\nabla d(f(x_0),p)\|\leq \lip(f)(x_0).$$ 
  To get the opposite inequality, we use a slight modification of the proof of \cite[Theorem 4.1.6]{a-t}.
Since $Y$  is separable, we fix  a countable dense set $\{p_n\}\subset Y$, then for every $x_1,x_2\in U$ we have 
\begin{equation}\label{sup}d(f(x_1),f(x_2))= \sup _n | d(f(x_1),p_n)-d(f(x_2),p_n)|.\end{equation}
Consider the Lipschitz function  $\varphi _n(t)=d(f(x_0+tv),p_n)$ and   set $m(t)=\sup _n |\dot \varphi _n(t) |$. Observe that $|m(t)|\leq \lip (f)$. By the Lipschitz condition, we may suppose that, for every $n\in \N$, $t=0$  is a differentiability point for $\varphi_n$  and  also $t=0$ is a Lebesgue point for $m \in L^\infty$. % (questo ï¿½ il punto da controllare maggiormente). 
By \eqref{sup} we obtain
$$\frac{d(f(x_0+tv),f(x_0))}{t} \leq \sup_n \frac{1}{t}\int _0^t|\dot \varphi _n(s) |ds\leq \frac{1}{t}\int_0^t m(s) ds .$$
Letting $t \rightarrow 0^+$ we get (see Prop. 1  and Th. 2 of \cite{kirc}) \begin{equation}\label{sup2}MD(f,x_0)(v)\leq m(0)\leq \sup_n \|\nabla d(f(x_0),p_n)\|.\end{equation}
On the other hand, by \eqref{m-diff} we get
$$ \frac{d(f(x),f(x_0))}{\|x-x_0\|}=MD(f,x_0)\left (\frac{x-x_0}{\|x-x_0\|}\right )+\frac{o(\|x-x_0\|)}{\|x-x_0\|}$$
Letting $x\rightarrow x_0$, by \eqref{sup2} we get
$$\lip(f)(x_0)\leq \sup_n \|\nabla d(f(x_0),p_n)\|.$$
\end{proof}
Observe that by the proof of the previous lemma, the following equality also holds true
\begin{equation}\label{md}
  \lip(f)(x_0)=\sup_{v\in \R^N,\; \vert v\vert =1} |MD(f,x_0)(v)|.
  \end{equation}
%~~~~~~~~~~~~~~~~~~~~
 \begin{lemma}\label{lsc}
Let $(Y,d)$ be a separable metric space. Assume $U\subset \R^N$ is an open set,   $(f_n)_{n\in\N}$ be a sequence
of (locally) equi-Lipschitz functions $f_n:U\rightarrow(Y,d)$  %$\forall n\in \N$
and let $f:U\rightarrow(Y,d)$.
 If $f_n\rightarrow f$ (locally) uniformly on $U$  then
\begin{equation} 
 \int_U\lip(f)(x) \ dx \leq \liminf _{n\rightarrow + \infty}\int_U \lip(f_n)(x) \ dx .\end{equation}
 \end{lemma}
 \begin{proof}
By uniform convergence   $f$ is a (locally) Lipschitz function, moreover  
 we have that 
$d(f_n(\cdot ), p) \rightharpoonup d(f(\cdot ), p)$ weakly$*$ in $W_{loc}^{1,\infty}(U)$.
Therefore, for every $p$, by weak l.s.c. of the gradient norm (see also Ch. III Th. 3.3 of \cite{b-d}) and using \eqref{suppo1} we get 
$$ \int _U \|\nabla d(f(x ), p)\| dx \leq \liminf_{n\rightarrow + \infty} \int _U \|\nabla d(f_n(x ), p )\| dx\leq \liminf_{n\rightarrow + \infty} \int _U \lip(f_n)(x) \ dx  .$$ 
Since $Y$ is separable, as in the proof of Lemma \ref{diff-lip},  denoting by $g_n=\|\nabla d(f(x ), p_n)\|$, by \eqref{suppo1} we may assume that $\lip(f)(x)=\lim_n g_n(x)$. Moreover, observe that 
$|g_n(x)|\leq \lip (f)$. Hence, passing to the limit under the integral sign we finally obtain 
$$ \int _U \lip(f)(x) \ dx = \lim _{n\rightarrow + \infty} \int _U g_n(x) dx 
\leq \liminf_{n\rightarrow + \infty} \int _U \lip(f_n)(x) \ dx .$$
 \end{proof}
 \begin{remark}
 The above Lemma holds true  of course for the function $\lip^p(f)$, for any $p\geq 1$. If $Y\subset \R^N$, the uniform convergence can be replaced by the weak convergence on the Sobolev space $W^{1,p}(U)$. In such a case, 
 Lemma~\ref{lsc}  just states  the lower semicontinuity of the $p$-Dirichlet energy in Sobolev spaces, since if $u:\R^N\rightarrow \R^N$ is differentiable at $x$ then 
$\lip(u)(x)=\|\nabla u(x)\|$ (see also  \cite[Ch. 3 Theorem 3.3]{b-d} and \cite{res-sob} for a related semicontinuity result).
 %Notice that if $U\subset \R^N$ is replaced by a general metric space $(X,d_X)$   the above arguments cannot be applied. In this case of generality we are forced to deal with metric Sobolev spaces (see Appendix~\ref{sobolev-metric-spaces}). 
 \end{remark}
%~~~~~~~~~~~~~~~~~~~~

\section{Reformation maps}
In this section we  introduce the class of {\em reformation maps} and establish some properties of these functions. Though the definition of reformation map holds for general metric measure spaces, as a first step we restrict our analysis to the euclidean framework of subsets of  $\R^N$.\\
Let $\Omega\subset\R^N$  be an open bounded connected set,  $X=\overline\Omega$,   
 $Y\subset \R^N$,     $\mu\in \P(X)$ and $\nu\in\P(Y)$. 

\begin{definition}[Expansion and contraction energy]
\label{c-e}
Let 
$x_0\in X$ and  $u:X\rightarrow Y$.  
The pointwise expansion  energy of $u$ at $x_0$ is  defined by

\begin{equation}
\label{expansion}
    e_u(x_0):=\lip(u)(x_0)= \limsup_
    {x\rightarrow x_0}
    \frac{|u(x)-u(x_0)|}{|x-x_0|}.
    \end{equation} 
    The pointwise contraction energy  of $u$ at $x_0$  is defined by
\begin{equation}
  c_u(x_0):= \limsup_
    {x\rightarrow x_0}
     \frac{|x-x_0|}{|u(x)-u(x_0)|}.
 \end{equation} 
 The pointwise reformation energy  of $u$ at $x_0$ is defined  by
\begin{equation}
r_u(x_0) =e_u(x_0)+ c_u(x_0).
\end{equation}
\end{definition}

%~~~~~~~~~~~~~~~~~~~~

\begin{definition}[Reformation maps]
\label{def-ref}
Given   $H,K:X\rightarrow  ]0,+\infty [$,  $H,K \in L^1(X, \mu)$ and a fixed covering $\mathcal A$ of $X$ made by balls, % $\subset \Cup_{x\in X, r>0}B(x,r)$ be flocally bounded, i.e. 
%$$ \forall x\in X \:\:\exists\: r>0\:\hbox{s.t.}\: K(y)\leq M(x)<+\infty ,\:\: H(y)\leq m(x)<+\infty  \:\:\forall y\in \overline B(x,r)\cap X. $$
  we define the set of reformation maps between $\mu$ and $\nu$, which we shall  denote by $\Refo(\mu;\nu)^{H,K}$, as the set of maps  
 $u:X\rightarrow Y$ 
such that the following conditions hold true:
\begin{equation}
\label{push}
  u_\#\mu=\nu,
 \end{equation} 
\begin{equation}
\label{bounds}
\forall x\in X \:\: \exists B(x,r)\in \mathcal A \: \: s.t. \: \: c_u(y)\leq H(x), \: \: e_u(y)\leq K(x) \: \: \forall y\in \overline B(x,r)\cap \Omega, 
%\exists\: r>0\:\hbox{s.t.}\:e_u(y)\leq K(x),\:\:c_u(y)\leq H(x)\:\:\forall y\in \overline B(x,r)\cap X.
\end{equation}
where $ u_\#\mu$ is the probability measure on $Y$ defined by $ u_\#\mu(A)=\mu(u^{-1}(A))$ for every Borel set $A$ of $Y$.
\end{definition}
 The point in the above definition is that the functions $c_u, e_u$ are locally bounded from the above by  $H(x),K(x)$ which may depend just on the point $x$ an not by the map $u\in \Refo(\mu;\nu)^{H,K}$.% or the radius $r>0$. 
 Therefore, the reformation maps  are characterized by  locally uniformly  bounded expansion and contraction.
Notice that, by the bounds \eqref{bounds}, any $u\in \Refo(\mu;\nu)^{H,K}$ is continuous and, by Stepanov  Theorem (see \cite{lip}),  
 is a.e. differentiable in $\Omega$. In particular, by Lemma \ref{lip-lemma} reformation maps are locally Lipschitz on $\Omega$. 
 \begin{remark}
 In a  mechanical perspective, the constraints stated in \eqref{bounds} could be considered as a bound on the maximum expansion or contraction experienced  by the material $\Omega$.  In this setting, the assumption that the bounds $H(x),K(x)$ do not depend on the map $u$ in \eqref{bounds} corresponds to  a constitutive property of the material under consideration.
 We point out that  bounds like \eqref{bounds} are in some sense necessary to control the geometry of the reformations. For instance, 
  in the case of $\nu=\delta_{y_0} $ we have $e_u=0$, $c_u=+\infty$ for any map $u$ satisfying \eqref{push}. On the other hand, mapping a bar into a bended one (see Fig. $2.1$) by two piecewise isometries $u_1$, $u_2$ such that $u_1(x_0)\neq u_2(x_0)$, we necessarily create a \emph{fracture} at the point $x_0$. It results $e_u(x_0)=+\infty$ at the discontinuity point $x_0$. See also Example \ref{ex}.
  Therefore, roughly speaking, the bound $c_u\leq H$ means  no implosion , while $e_u\leq K $ means no fractures.
 \begin{figure}[htbp]\label{fig:ben-bar}
      % \begin{pspicture}(-3.5, -1.5)(6.5, 2.5)\psline(-1,-1)(-1,1)\psline(2,-1)(2,0)(3,0)\psdots(-1,0)(2,0)\uput*[0](-1,0){ $x_0$ }\uput*[0](0.7,0){$u(x_0)$}\psline{->}(-0.5,-0.5)(1.5,-0.5)\pscurve{->}(-0.5,1)(0.5,1)(2,0.2)\uput*[0](0.3,-0.8){$u_1$}\uput*[0](0.5,1.3){$u_2$}\end{pspicture}
      \includegraphics[scale=0.8]{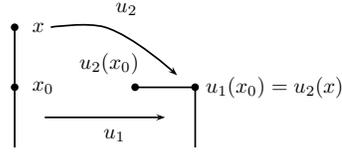}
\caption{Mapping a bar into a bended one.}
\end{figure}
 \end{remark}
 %%%%%%%%%%%%%%%%%%%%%%%%%%%%%%%%%%%%%%%%%%%%%%%%%%
 \begin{remark}\label{inversion}
 The constraint $c_u\leq H$ in \eqref{bounds} is related to inversion properties, both local or global, of reformation maps, see 
 \cite{garrido-global, metric-local, john}. Observe that for differentiable maps with non-vanishing Jacobian we always have  %(see \cite[Lemma 2]{inversion})
 $$c_u=\|(\nabla u) ^{-1}\|, \quad e_u=\|\nabla u \|.$$  %Compare also with Lemma \ref{det}.
 
 Therefore, it would be an interesting question to consider just pointwise bounds $H(x),K(x)$. This is similar in spirit to the passage from functions with bounded distortion to functions with finite distortion (see the monograph \cite{iwaniec}). However, in such a case, inversion properties becomes more subtle and further assumptions are needed, see for instance \cite{inver-sob-distortion, inverse-sobolev, putten} for inversion results of Sobolev maps. Anyway, the main interest of reformation maps relies in this perspective in considering just metric objects (see Section $4$ below).
 %for an extension in a metric setting. 
 However, in a purely metric framework, 
 %also uniform bounds as in \eqref{} 
 such pointwise conditions are not enough to guarantee inversion properties. Consider for instance the map $u:\R \rightarrow \R, u(x)=|x|$ having $e_u=c_u=1$ at every point. Therefore, also uniform bounds alone are not enough to get satisfactory inversion properties. Under differentiability assumptions in $\R^N$,  pointwise bounds are in fact enough, see Lemma \ref{det}. However, in general  this is not true.  For instance (see \cite{inversion}),  it is possible to find everywhere differentiable maps with everywhere invertible differential on Hilbert spaces which are not neither open or locally one-to-one. 
 In the metric setting different restrictions arise (see \cite{metric-local} for a detailed discussion). %we don't know if also uniform bounds as in \eqref{bounds} could be guarantee satisfactory inversion properties, see the discussion in Section $4$.
 \end{remark}
 %%%%%%%%%%%%%%%%%%%%%%%%%%%%%%%%%%%%%%%%%%%%%%%%%%
 \begin{remark}\label{quasi}
 The local uniform bounds in Def. \ref{def-ref} are also related to quasi-isometries, see for instance \cite{ nonlinear,john}. In such case uniform bounds 
 $$m\leq D^-f(x)\leq D^+f(x)\leq M, $$
 where
 $$  D^-f(x_0)=\liminf_{x\rightarrow x_0}\frac{|f(x)-f(x_0)|}{|x-x_0|}, \qquad D^+f(x_0)=\limsup_{x\rightarrow x_0}\frac{|f(x)-f(x_0)|}{|x-x_0|}$$
 denote the local distortion of distances, are required. Observe that $e_u(x)= D^+f(x)$, while $c_u(x)= \frac{1}{D^-f(x_0)}$. Indeed, 
 $$ \frac{|x-x_0|}{|f(x)-f(x_0)|} = \frac{1}{\frac{|f(x)-f(x_0)|}{|x-x_0|}} \leq \frac{1}{\inf _B\frac{|f(x)-f(x_0)|}{|x-x_0|}}.$$
 Taking the supremum over $B=B(x_0,r)$ and then letting $r\rightarrow 0^+$ we get $c_u(x)\leq \frac{1}{D^-f(x_0)}$.
 In a similar way the opposite inequality follows. 
 
 \end{remark} 
 
 %%%%%%%%%%%%%%%%%%%%%%%%%%%%%%%%%%%%%%%%%%%%%%%%%%%%%%%%%%%%%%%%
 
 \begin{remark}
 The mass conservation property \eqref{push} is a generalized version of incompressibility  and it can be always satisfied (provided $\mu$ has no atom, see \cite{pratelli}) by some measurable map $u$. Actually, condition \eqref{push} is equivalent to the following change of variable formula
 \begin{equation}\label{push-eq} 
 \int_X f(u(x))\ d\mu =\int_Y f(y)\ d\nu, \end{equation} for every continuous function $f:Y\rightarrow \R$. 
In the setting of Mass Transportation, maps $u$ satisfying \eqref{push} are called \emph{transport maps}.
% and for a given cost function $c:X\times Y\rightarrow[0,+\infty[$, the Monge mass transportation problem amounts to
%\begin{equation}\label{mongw} \min \left \{ \int_X c(x,t(x)) \ d\mu \ : \ u_\# \mu = \nu \right \}.\end{equation}
 \end{remark}
 \begin{remark}
  Since   $\Refo(\mu;\nu)^{H,K}$ is made by \emph{nice} functions, 
 to compare the present approach with other  ones, as for instance that of  \cite{wolan}, 
 one should also require the condition $u(X)=Y$ in the case of $\mu =\L^N\res X$, $\nu =\L^N\res Y$. However, we point out that this surjection requirement is actually a severe constraint. Indeed,  for instance, to find   
 Lipschitz functions  $u:X\rightarrow Y, \ u(X)=Y$ for general compact sets in dimension  $N\geq 3$  (see \cite{alberti}), as far as we know, is still an open question. 
 Moreover, also bi-Lipschitz functions between \emph{nice} sets are not easy to find (see \cite{fonseca-parry, bilip}). 
 Anyway, in such a case many restrictions  on the target space $Y$ may be needed  (connectedness, for instance).
 \end{remark}
 In the following we prove some properties enjoyed by reformation maps. A first estimate easy to verify is an immediate consequence of  
 \eqref{bounds} and is given by the following proposition.
 
  \begin{prop}
  \label{constraint}
  Let $u\in \Refo (\mu;\nu)^{H,K}$. Then, for every  $ x_0 \in X$  there exists   $r>0$ 
 such that
 \begin{equation}
 \label{rub} 
 \ \frac{1}{H(x_0)}|x-x_0|\leq |u(x)-u(x_0)|\leq K(x_0)|x-x_0|\:\:\:\: \:\forall \ \!x\!\in\! X\cap \overline B(x_0,r).
 \end{equation} 
 \end{prop}

 \begin{lemma}
 \label{discreta}
 Let $u\in \Refo(\mu;\nu)^{H,K}$. Then $u$ is a discrete map, i.e. for every $y\in Y$, $u^{-1}(y)$ is a finite set.
 \end{lemma}
 \begin{proof}Let $y\in Y$. By \eqref{rub}, if $x_0\in u^{-1}(y)$ we have that $u(x)\neq u(x_0)$ for every $x\in \overline B(x_0,r)$. Hence, $x_0$ is an isolated point of $u^{-1}(y)$. We claim that $u^{-1}(y)$ is a finite set. Indeed, otherwise, since $X$ is compact, we find a sequence $x_n\rightarrow x_0 \in X$ such that $x_n\in u^{-1}(y)$. By continuity of $u$ we have  also $x_0\in u^{-1}(y)$. This is a contradiction since $x_0$ is  an isolated point of $u^{-1}(y)$.
 \end{proof}

 For discrete continuous  maps  the \emph{local degree} or \emph{local index}  $i(x_0,u)$
 of $u:X\rightarrow Y$ at $x_0\in X$ is  defined  
 as follows
 \begin{equation}
 i(x_0,u)={\rm deg}(u(x_0),u,\overline B(x_0,r)), 
  \end{equation}
 where ${\rm deg}(y,u,B)$ denotes the topological degree (see  \cite{f-g, b-d}).\\
Let  $u^{-1}(y)=\{x_1,\cdots, x_h \}$, we have the following relation
 \begin{equation}\label{sum} {\rm deg}(y,u,Y)=\sum_{j=1}^ h i(x_j,u).
 \end{equation} 
   Observe that if $u$ is locally injective in a neighborhood of $x\in X$
    then $|i(x,u)|=1$.\\
  We say that $u$ is a \textit{sense-preserving (reversing) continuous map} if the local index $i(x,u)$ has constant sign  in $X$.
  Notice  that each homeomorphism on a domain is either sense-preserving or sense-reversing (see  \cite[Theorem 3.35]{f-g}). Moreover,  sense-preserving or sense-reversing  differentiable maps have   constant Jacobian sign,  (see   \cite[Lemma 5.9]{f-g}), since
 \begin{equation}
 \label{jac}
 i(x_0,u)=sign(Ju(x_0)), 
 \end{equation} where $Ju:={\rm det} \nabla u$, 
 providing   $Ju(x_0)\neq 0$.\\

\begin{lemma}
 \label{det} 
 Let 
 $u\in \Refo(\mu;\nu)^{H,K}$.  If $u$ is  differentiable at $x_0 \in \Omega$  then   $J u(x_0)\neq 0$.
 \end{lemma}
 \begin{proof}
 Suppose by contradiction that $J u(x_0)= 0$. Then  we find a vector $|v|=1$ such that $\nabla u(x_0) \cdot v =0$. Fixed $\eps >0$, since $u$ is differentiable, there exists $\delta >0$ such that $|u(x_0+tv)-u(x_0)|< \eps t$, whenever $|t|<\delta $. On the other hand, by \eqref{rub}, there exists $0<t<\delta$ such that $\frac{t}{H}<|u(x_0+tv)-u(x_0)|< \eps t$. Hence $\frac{1}{H}<\eps$.  Letting $\eps \rightarrow 0^+$ we get a contradiction.
 \end{proof}
 %\begin{remark}\label{inver-diff}
 By the previous lemma,  any $u\in C^1(X;Y)\cap \Refo(\mu;\nu)^{H,K}$ (or even an everywhere differentiable reformation  map)  is locally invertible on $\Omega$  
 (see \cite{inverse} for a related inversion result and \cite{inversion} for an elementary analytical proof).\\
If $u$ is only  a.e. differentiable, by Lemma \ref{det} we  have 
 $J u(x)\neq 0$  for a.e. $x$. However,  it is well known that in general   this condition does  not ensure  the local invertibility of Sobolev maps 
 (see for instance \cite{inver-sob-distortion}).   
By the way, the condition
   $Ju>0$  on an open set $\Omega$ is a standard requirement (see for  instance   \cite{putten}), ensuring   that $u$ is locally invertible for a.e. $x\in \Omega$.   The restriction to sense-preserving maps is also made in \cite{metric-local} to derive local inversion results. To this aim, also assumptions on the boundedness of $HK\leq M$, for sufficiently small $M$ are necessary.
   In our context, since we are interested in comparing domains, also in a metric framework, restrictions to open maps seem more  natural. 
   
   %For reader convenience 
   We refer  \cite{implicit} for a proof of  the following  result. 
%\end{remark} 
\begin{theorem}\label{open-discrete}
Let $u: \Omega \rightarrow \R ^N$ be a continuous open and discrete  map. Then $u$ is sense-preserving or sense-reversing.
\end{theorem}

   \begin{lemma}\label{injectivity}
   Let $u: \Omega \rightarrow \R ^N$ be a discrete  sense-preserving (reversing) continuous map such that $|i(x_0,u)|=1$. Then $u$ is injective in a neighborhood of $x_0$.
 \end{lemma}
 \begin{proof}We use here the same arguments of \cite[ Ch. II Theorem 6.6 ]{b-d}. 
 Suppose  $i(x_0,u)=1$. The other case is analogous. 
 By contradiction, if $u$ is not injective we  have two distinct  sequences $(x_n^1)_{n\in \N}$, $(x_n^2)_{n\in \N}$ converging to $x_0$ such that
 for every $n\in\N$:  $u(x_n^1)=u(x_n^2)=y_n$. By continuity we also have
 $u(x_n^1)\rightarrow y_0=u(x_0)$. Since the degree is constant in a neighborhood of $y_0$, for  $n\in\N$ large enough and suitable small radius $r>0$ we have %(see Prop. of \cite{})
 \begin{equation}
 \label{deg1} 
 {\rm deg}(y_n,u,\overline B(x_0,r))={\rm deg}(y_0,u,\overline B(x_0,r))= i(x_0,u)=1.
 \end{equation}
 On the other hand, since $u$ is  sense-preserving and by \eqref{sum} we have
 $$ {\rm deg}(y_n,u, \overline B(x_0,r))\geq i(x_n^1,u)+i(x_n^2,u)=2, $$
 contradicting \eqref{deg1}.
  \end{proof}
  %%%%%%%%%%%%%%%%%%%%%%%%%%%%%%%%%%%%%%%%%%%%%%%%%%%%%%%%%%%%%%%%%%%%%%%%%
%We get the following inversion properties.
\begin{theorem}[Invertibility of incompressible maps]\label{incompressible}
Assume $u\in \Refo(\mu;\nu)^{H,K} $ is an \textit{incompressible map}  (see for instance \cite{wolan}), i.e.
 \begin{equation}
 \label{preserving}
  |Ju(x)|=1  \ \mbox{for }  a.e.\  x\in \Omega .
 \end{equation} 
Then $u_{| \Omega} $ is globally invertible.
\end{theorem}
\begin{proof}
Of course, condition \eqref{push} holds true for injective such maps $u$. Anyway, let us begin by showing that $u$ is an open map. 
Let $U$ be an open subset of $\Omega$. We have to prove that $V=u(U)$ is open. Let $y_0\in V$, $y_0=f(x_0)$ for a $x_0 \in U$. By  \eqref{rub}
  we find $C=\overline B(x_0,r)$ such that $y_0 \notin u(\partial C)$. Therefore, it is well defined ${\rm deg}(y_0,u,C)$ which is constant in a neighborhood of $y_0$.  If $u$ is differentiable at $x_0$,   by Lemma~\ref{discreta} and Lemma~\ref{det} we have 
  $$|{\rm deg}(y,u,C)|=|{\rm deg}(y_0,u,C)|=1,$$ in a neighborhood $V_{y_0}$ of $y_0$. Since ${\rm deg}(y,u,C)\neq 0$, it follows that  $V_{y_0}\subset u(C)\subset V$, since otherwise the degree would be zero.\\
  %Observe that  since $\chi_B (x) \leq \chi _{u(B)}(u(x))$, by Lusin Theorem and \eqref{push-eq},  we have \begin{equation}\label{push1} \L^N(B )=\int _X \chi_B (x) \ dx \leq  \int _X \chi_{u(B)}(u(x))\ dx = \int _ Y \chi_{u(B)}u(y)\ dy =\L^N(u(B )). \end{equation}
  On the other hand, denoting by $N(y,u,K):={\rm card}\{u^{-1}(y)\cap K\}$ the multiplicity function, by the Area Formula and the push-forward condition \eqref{push-eq} we compute
  $$\L^N(u(B(x_0,r)))\leq \int _{u(B(x_0,r))} N(y,u,B(x_0,r))\ dy=\int_{u^{-1}(u(B(x_0,r)))} |Ju|\ dx = $$
  $$ \L^N(u^{-1}(u(B(x_0,r))))=\L^N(u(B(x_0,r))).$$

 %$$ \!0\leq\! \int _{\R^N}\!\! \left ( N(y,u,B(x_0,r)) -\chi_{u(B(x_0,r))}\right ) dy \!=\! \int_{B(x_0,r)}\!\!\! |Ju(x)|\ dx -\L^N(u(B(x_0,r)))= $$ $$ =\L^N(B(x_0,r))- \L^N(u(B(x_0,r)))\leq 0, $$
  Therefore, for a.e. $y\in u(B(x_0,r))$,  it results $N(y,u,B(x_0,r))=1$. Hence, if $x_0$ is a non-differentiability point, by the Lusin (N)-property, there exists a differentiability point $x\in B(x_0,\delta)$ of $u$  and an open neighborhood $V_{y_0}$ of $y_0$ such that 
 $u(x)\in V_{y_0}$ and ${\rm deg}(u(x),u,C)\neq 0$. As above %in the proof of the previous lemma, 
 it follows that $V_{y_0}\subset u(C)\subset V$. Hence $u_{| \Omega} $ is open.
 By Theorem \ref{open-discrete} we have that $u$ is sense-preserving, or sense-reversing. Since actually it results $|i(x_0,u)|=1$, %as shown in the proof of Lemma~\ref{open}, the result 
 the statement follows by Lemma~\ref{injectivity}. 
\end{proof}
\begin{remark}
  If a  reformation map satisfies \eqref{preserving}, then $|i(x,u)|=1$ for a.e. $x\in  \Omega $. Therefore, by Lemma \ref{injectivity} such map $u$ is a.e. locally invertible on $\Omega$. As previously discussed, actually to obtain this invertibility property it is enough to require
  $|Ju|\leq 1 $ a.e. in $\Omega$. Since $|Ju|\leq \|\nabla u \|^N$, this happens for instance for reformation maps $u$ satisfying the condition  $e_u\leq 1$.
  \end{remark}
%%%%%%%%%%%%%%%%%%%%%%%%%%%%%%%%%%%%%%%%%%%%%%%%%%%%%%%%%%%%%%%%%%%%
 \begin{theorem}[Invertibility of small reformations]\label{invertibility1}
 Let $u\in \Refo(\mu;\nu)^{H,K} $ be such that $e_u(x) < \sqrt[N]{2}$ for a.e. $x$.  Then $u $ is globally invertible. 
 \end{theorem}
 \begin{proof}
  %\label{invertibility1}
  Recalling the assumption that both $\mu$ and $\nu$ have density given by  characteristics functions of an open set, by the constraint $c_u\leq H$ we find (see \cite[Prop. 1.1, Sec. 3]{open-metric}) an open dense subset $U\subset X$ on which $u$ is locally bi-Lipschitz.  
  It follows that   the multiplicity function $N(y,u,U)$ is locally constant (see \cite{ambrosetti-prodi}). We first prove that $u$ is globally invertible on $U$.
  For $y=u(x)\in u(U)$, we  prove that $N(y,u,U)=1$. Observe that by the Domain Invariance  Theorem (see for instance \cite{f-g}), $u_{| U} $ is open. Let $B=B(x,r)$ be a ball on which $u$ is bi-Lipschitz. We may suppose that $D=N(y,u,B)$ is constant on $u(B)$. 
   Since $|Ju|\leq \|\nabla u\|^N$,  
  by the Area Formula we compute
   \begin{equation}\begin{split}\label{d}D\L^N(u(B))&= \int_{u(B)}N(y,u,B ) dy = \int_{u^{-1}(u(B))} |Ju(x)| dx \\
  &\leq \int _{u^{-1}(u(B))} \|\nabla u \|^N dx =\int_{u^{-1}(u(B))} e_u(x)^N dx.\end{split}
   \end{equation} 
  %Observe that by the push-forward condition, as in \eqref{push1},  we obtain $\L^N(B)\leq \L^N(u(B))$.
  %\begin{equation}\label{d1} \L^N(\Omega )=\int _X \chi_\Omega (x) \ dx \leq  \int _X \chi_{u(\Omega)}(u(x))\ dx = \int _ Y \chi_{u(\Omega)}u(y)\ dy =\L^N(u(\Omega )). \end{equation}
  Therefore,  if the map $u$ satisfies the   small expansion condition $e_u<\sqrt[N]2$, % we have
  %\begin{equation}\label{small}  \int_{B} e_u(x)^N \ dx<2\L^N(B ),  \end{equation}
   by  \eqref{d} and the push-forward condition we get
  $$ D\L^N(u(B )) <2\L^N(u^{-1}(u(B)))= 2 \L^N(u(B) ),$$ hence  the map $u$ is globally invertible on $U$. 
  %This  happens, for instance, for  reformation    maps satisfying  the condition    \begin{equation}    e_u<\sqrt[N]{2}.  \end{equation}
 By uniform continuity, $u$ uniquely extends to the whole $X$ and therefore letting to global invertibility on $X$. 
 \end{proof}
 %%%%%%%%%%%%%%%%%%%%%%%%%%%%%%%%%%%%%%%%%%
 %\begin{remark}  By  Area Formula,    for $B=B(x_0,r)$ using \eqref{d} it results 
  %$$ \L^N(u(B))\leq \int _{\R^N} N(y,u,B) dy=\int_{\R^N} |Ju|\ dy \leq  \int_{\R^N} \|\nabla u \|^N\ dy =\int_{\R^N} |e_u|^N\ dx.$$
  %~~~~~~~~~~~~~~~~~~~~ 
  %Therefore, 
  %$$ \mean {B}e_u(x)^N dx  \geq \frac{\L^N(u(B))}{\L^N(B)}.$$    By Lebesgue Differentiation Theorem this yields the a.e. inequality   $$ e_u(x_0)\geq\left ( \lim_{r\rightarrow 0^+} \frac{\L^N(u(B))}{\L^N(B)}\right )^\frac{1}{N}:=\mu_u(x_0)^\frac{1}{N}.$$  The volume derivative $\mu_u(x_0)$ is also related to \emph{quasiconformal maps}. We refer to  \cite{lip, koskela}, see also \cite{gevirtz-ball}.  \end{remark}
 %%%%%%%%%%%%%%%%%%%%%%%%%%%%%%%%%%%%%%%%%%%%%%%%%%%%%%%%%%%%%%
 \begin{theorem}\label{metric-inversion}
 Let $u\in \Refo ^{H,K}(\mu;\nu) $ be an open map  and suppose that   $\forall x \in X : H(x)K(x)<2$. Then  $u_{| \Omega} $ is locally invertible.
 \end{theorem}
 \begin{proof} By Lemma \ref{discreta} and Theorem \ref{open-discrete} it follows that $u$ is sense-preserving or sense-reversing. Hence, by Theorem II of  \cite{metric-local} the thesis follows.  \end{proof}
 The condition $e_uc_u\leq \alpha $ %  can be enlarged to $2$ by requiring $e_u c_u<2 $ to hold everywhere. 
  may be required to hold just  a.e. by reducing the upper bound $\alpha <\sqrt[N]{2}$ 
 (see \cite{metric-local}). %Observe that the results of \cite{metric-local} hold for dimension $N\geq 2$. 
 Observe that the map $u(x)=|x|$ in one dimension is not a counterexample to Theorem \ref{invertibility1} since $u$ is not mass preserving around $x_0=0$. 
 If $\Omega$ is a ball, or in some classes of convex sets, for sufficiently small $\alpha$ the map $u$ is actually globally invertible (see \cite{gevirtz}). 

\section{The variational problem of elastic reformation}
\begin{definition}
Let $u:X\rightarrow Y$, $\mu\in\P(X)$, $\nu\in \P(Y)$, such that  $u_\#\mu=\nu $.
 We define the  total reformation energy $\mathscr{R}(u)$ of a  reformation map $u$ of $\mu$ into $\nu$ as follows
\begin{equation}
\label{reformation} 
\mathscr{R}(u):= \int _X r_u(x) \ d \mu .
\end{equation}
\end{definition}
%%%%%%%%%%%%%%%%%%%%%%%%%%%%%%%%%%%%%%%%%
We recall that  $\mathscr{R} (u)<+\infty$ for every $u\in \Refo(\mu;\nu)^{H,K}$, since we will always assume 
\begin{equation}
  H(x),\;K(x)\in L^1(X,\mu),
 \end{equation} 
  where $H,K$ are given in Definition~\ref{def-ref}. 
We have the following
\begin{lemma}
\label{maggio}
Let $u:X\rightarrow Y$. Then 
\begin{equation}
  \label{hom1}
c_u(x)\geq \frac{1}{e_u(x)}\:\:\:\forall x\in X.
\end{equation}
\end{lemma}
\begin{proof}It suffices to recall that $c_u(x)=\frac{1}{D^-f(x)}$, see Remark \ref{quasi}.
%For $y\in B_r:= B(x,r)$ we compute$$ \frac{|y-x|}{|u(y)-u(x)|}= \left({\frac{|u(y)-u(x)|}{|y-x|}}\right)^{-1}\geq \frac{1}{\sup _{z\in B_r}\frac{|u(z)-u(x)|}{|z-x|}}.$$Taking  the supremum with respect to $y\in B_r $ and letting $r  \rightarrow 0^+$, we get (\ref{hom1}).
\end{proof}
Observe that for every $u:X\rightarrow Y$ such that $u_\#\mu=\nu $, by Lemma \ref{maggio} it results
$\mathscr{R}(u)\geq 2$.\\
 Actually,    Definition~\ref{reformation}  is motivated by the trivial fact that the real function 
 $f(x)= x+1/x$ reaches its minimum value at $f(1)=2$. Moreover,
  observing that  at any $x_0\in X$
   \begin{equation}
  \label{s1}
  r_u(x_0)\geq e_u(x_0)+ \frac{1}{e_u(x_0)}\geq 2, 
  \end{equation}  
  we have that   $r_u(x_0)$ reaches its minimum value if $u:X\rightarrow Y$ is an isometric mapping, i.e. $\vert u(x)-u(y)\vert=\vert x-y\vert$, $\forall x,y \in X$. Therefore   $\mathscr{R}(u)$  can be viewed as a measure detecting how $u$ is far from being an isometric map.

%A useful property is the following
\begin{lemma}\label{reform}
Let $u:X\rightarrow Y$ be   a local homeomorphism. Then 
\begin{equation}
\label{hom}
  c_u(x)=e_  {u^ {-1 } }(u(x))\:\:\:\forall x\in X. 
  \end{equation}
  %\begin{equation}  \label{hom1}c_u(x)\geq \frac{1}{e_u(x)}\:\:\:\forall x\in X.\end{equation}
\end{lemma}
\begin{proof}
Fix $x\in X$, $\delta_1>0$ and let  $B_1=B(u(x),\delta _1)$. By the local homeomorphism condition, there exists a $\delta >0$ such that $u(B_\delta )\subset B_1$ and 
 $u$ is invertible on  $B_\delta = B(x, \delta )$. For every $y\in B_\delta $ we have
$$ \frac{|y-x|}{|u(y)-u(x)|}=\frac{|u^{-1}(u(y))-u^{-1}(u(x))|}{|u(y)-u(x)|}\leq \sup _{z\in B_1} \frac{|u^{-1}(z)-u^{-1}(u(x))|}{|z-u(x)|}.$$
Taking  the supremum with respect to $y\in B_\delta $ and letting 
$\delta _1  \rightarrow 0^+$, we get $c_u(x)\leq e_{u^{-1}}(u(x))$. Analogously we deduce  the opposite inequality. 
%For $y\in B_r:= B(x,r)$ we compute $$ \frac{|y-x|}{|u(y)-u(x)|}= \left({\frac{|u(y)-u(x)|}{|y-x|}}\right)^{-1}\geq \frac{1}{\sup _{y\in B_r} \frac{|y-x|}{|u(y)-u(x)|}}.$$Passing to the supremum with respect to $y\in B_r $ and letting $r  \rightarrow 0^+$, we get (\ref{hom1}).
\end{proof}
%%%%%%%%%%%%%%%%%%%%%%%%%%%%%%%%
%Observe that for every $u:X\rightarrow Y$ such that $u_\#\mu=\nu $, by Lemma \ref{maggio} it results$\mathcal R(u)\geq 2$.\\ Actually,   Definition~\ref{reformation}  is motivated by the trivial fact that the real function  $f(x)= x+1/x$ reaches its minimum value at $f(1)=2$. Moreover, observing that  at any $x_0\in X$   \begin{equation}  \label{s1} r_u(x_0)\geq e_u(x_0)+ \frac{1}{e_u(x_0)}\geq 2,   \end{equation}    we have that  $r_u(x_0)$ reaches its minimum value if $u:X\rightarrow Y$ is an isometric mapping, i.e. $\vert u(x)-u(y)\vert=\vert x-y\vert$, $\forall x,y \in X$.Therefore   $\mathcal R(u)$  can be viewed as a measure detecting how $u$ is far from being an isometric map.

\begin{definition}
\label{elastic} 
We define the elastic reformation energy between $\mu$ and $\nu$ as
\begin{equation}\label{inf}\mathcal E(\mu,\nu):= \inf \{ \mathscr{R}(u)\:\vert 
\: u\in \Refo(\mu;\nu)^{H,K}
%u_\#\mu=\nu 
\}.
\end{equation}
\end{definition}
%Since the Monge transport problem is not symmetric, 
In general, the above elastic reformation energy is not symmetric. For instance, if $\mu =\L^N\res B$, for a ball $B$, and $\nu$ a Dirac delta,  we have $\mathcal E(\mu,\nu)=+\infty$. Reversing the shapes, we see that $\mathcal E(\nu,\mu)$ has no meaning simply because $\Refo(\nu;\mu)^{H,K}=\emptyset$. Moreover, also in nice cases, the matter is that reformation maps could be not invertible. Assuming invertibility 
for $u\in \Refo(\mu;\nu)^{H,K}$, setting $v:=u^{-1}$, by using Lemma \ref{reform} we have 
$$ \mathscr{R}(u)=\int_X e_u\  d\mu + \int_X c_u\ d\mu = \int_X e_u(u^{-1}(u(x)))\ d\mu +\int_X e_{u^{-1}}(u(x)) \ d\mu =$$ $$
\int_Y e_{v^{-1}}(v(y))\ d\nu +\int_Y e_v(y)\ d\nu = \int_Y c_v(y)\ d\nu +\int_Y e_v(y)\ d\nu= \mathscr{R}(v).$$
Since $v\in \Refo(\nu;\mu)^{H,K}$, we get $\mathcal E(\mu,\nu)=\mathcal E(\nu,\mu)$. Therefore, symmetry issues  essentially correspond to invertibility of maps.

The question is now to establish  conditions to ensure  the infimum in \eqref{inf} is attained.  
It is easily seen that 
\begin{equation}
\mathscr{R}(u)=2 \hbox{ \rm if and only if } r_u(x)=2\:\:\:\:\hbox{ \rm for } \mu-a.e.\: x\in X.
\end{equation} 
 %In order  to select isometric maps through a variational property, we need to characterize the maps $u:X\rightarrow Y$ such that  $r_u(x_0)=2$. 
\begin{lemma}
\label{iso}
Let $x_0\in X$, $u:X\rightarrow Y$. Then 
$r_u(x_0)=2$     if and only if
$$\forall \eps >0: \frac{1}{1+\eps}|x-x_0|\leq |u(x)-u(x_0)|\leq(1+\eps)|x-x_0|,\:\:\:\forall x\in X\cap B(x_0, r_\eps).$$  
\end{lemma} 
\begin{proof}
%f $t$ is a local isometry in $x_0$ then $e_t(x_0)=c_t(x_0)=1$. On the other hand, 
Assume $r_u(x_0)=2$, then 
$$ 2= e_u(x_0)+ c_u(x_0) \geq e_u(x_0)+ \frac{1}{e_u(x_0)} \geq 2, $$
 so
$$  e_u(x_0)+ \frac{1}{e_u(x_0)}=2 \Rightarrow (e_u(x_0)-1)^2=0 \Rightarrow e_u(x_0)=c_u(x_0)=1.$$
Fix $\eps>0$, then  $e_u(x_0)< 1+\eps $ implies that $u$   satisfies  
\begin{equation}
\label{lip} 
|u(x)-u(x_0)|\leq (1+\eps)|x-x_0|, \:\:\: \forall x\in X\cap B(x_0, r_\eps).
\end{equation} 
 By using the condition $c_u(x_0)<1+\eps $, eventually by decreasing the radius $r_\eps$, we get the opposite inequality.  
Vice versa,  if both the inequalities   locally hold, then it results $2\leq r_u(x_0)=e_u(x_0)+c_u(x_0)\leq 1+1=2$.
\end{proof}
Therefore, the  maps $u:X\rightarrow Y$ such that  $r_u=2$ are in some sense \emph{pointwise} locally quasi-isometric, (see \cite{quasi-isometric} for the relation with quasi-conformal maps). \\
In the following   we shall try  to characterize in a more precise way the reformation maps 
$u\in \Refo(\mu;\nu)^{H,K}$, 
if any,  realizing the minimum  energy level  $\mathscr{R}(u)=2$. 
We also want to prevent pathological situations as
the one  described  in Example \ref{ex} below in which the map $u:X\rightarrow Y$  is merely a.e. continuous  
(it is actually a.e. invertible and differentiable). \\
In particular, by \eqref{s1} (see also the proof of Lemma \ref{iso}) it results   $e_u(x)=1$ for  $\mu$-a.e. $x\in X$. 
Moreover,  $e_u(x)<+\infty$ implies $u$ continuous  at $x$. 
Then  these reformation maps $u$ are at least a.e. continuous functions.  
However, this mild regularity is too poor to preserve  geometric (or physical)
 properties as we show in the next example.
\begin{example}\label{ex}
Let $\Omega\subset \R^N$ be a smooth bounded open set and $Q\subset \R^N$ be a cube  such that $\L^N( \Omega)=\L^N( Q)$ (see figure $3.1$). 
For $n\geq 1$ large enough, $\Omega$ contains a certain number of disjoint  squares $Q_n$ of length $\frac{1}{n}$. Then consider  the map $u_n$ which isometrically moves every square $Q_n$ inside $Q$ in a disjoint way. On the remainder of $\Omega$, consider the contained squares $Q_m$, $m>n$, and then the map $u_m$ which coincides with $u_n$ on the squares $Q_n$ and moves by an isometry  the squares $Q_m$ inside $Q$ in a disjoint way. By this procedure it is then defined a sequence $(u_n)_{n\in \N}$. Taking the limit $u=\lim_{n\rightarrow +\infty} u_n$ we obtain a measurable map $u  :\Omega \rightarrow Q$ such that $u_\# \mu = \nu$, where $\mu =\mathcal L^N\res \Omega$, $\nu=\mathcal L^N\res Q $, and $r_u(x_0)=2 $ for  a.e. $x_0\in \Omega$. 
Therefore, every bounded smooth open set can be reformed into a square at minimal energy. 
%La funzione $u$ dovrebbe essere BV.
\end{example}
%%%%%%%%%%%%%%%%%%%%%%%%%%%%%%%%%%%%%%%%%%%%%%%%%%
\begin{figure}\label{ex-fig} 
\begin{pspicture}(-5.5, -1.5)(8.5, 3.5)
%\psgrid[subgriddiv=1,griddots=10,gridlabels=7pt](-5,-3)(5,3)
\psline(2,-1)(4.5,-1)(4.5,1.5)(2,1.5)(2,-1)
\pscircle(-2,0){1.45}
\psline(-2,-1)(-1,-1)(-1,0)(-2,0)(-2,-1)
\psline(-2,-1)(-3,-1)(-3,0)(-2,0)
\psline(-3,0)(-3,1)(-2,1)(-2,0)
\psline(-2,1)(-1,1)(-1,0)
%%%%%%%%%%%%%%%%%%%%%%%%%%%%%
\psline(2,-1)(3,-1)(3,0)(2,0)
\psline(2,0)(2,1)(3,1)(3,0)
\psline(3,1)(4,1)(4,0)(3,0)
\psline(3,0)(4,0)(4,-1)
%%%%%%%%%%%%%%%%%%%%%%%%%%
\psline(-2,1)(-2,1.35)(-1.65,1.35)(-1.65,1)
\psline(-2,1.35)(-2.35,1.35)(-2.35,1)
\psline(-3,0)(-3.35,0)(-3.35,0.35)(-3,0.35)
\psline(-3.35,0)(-3.35,-0.35)(-3,-0.35)
\psline(-2,-1)(-2,-1.35)(-2.35,-1.35)(-2.35,-1)
\psline(-2,-1.35)(-1.65,-1.35)(-1.65,-1)
\psline(-1,0)(-0.65,0)(-0.65,0.35)(-1,0.35)
\psline(-0.65,0)(-0.65,-0.35)(-1,-0.35)
%%%%%%%%%%%%%%%%%%%%%%%%%%%%%%%%%%
\psline(2,1)(2,1.35)(2.35,1.35)(2.35,1)
\psline(2.35,1.35)(2.70,1.35)(2.70,1)
\psline(2.70,1.35)(3.05,1.35)(3.05,1)
\psline(3.05,1.35)(3.40,1.35)(3.40,1)
\psline(3.40,1.35)(3.75, 1.35)(3.75,1)
\psline(3.75,1.35)(4.05,1.35)(4.05,1)(4,1)
\psline(4.05,1.35)(4.40,1.35)(4.40,1)(4.05,1)
\psline(4,1)(4.35, 1)(4.35,0.65)(4,0.65)
%%%%%%%%%%%%%%%%%%%%%%%%%%%%%%%%%%%%%%%%%%%%%
\pscurve{->}(-1,1.5)(0.5,2.3)(2.5,1.8)

\uput*[0](0.3,1.7){ $u_n$ }
\end{pspicture}
\caption{A piece-wise isometric map from the circle into a square.}\end{figure}
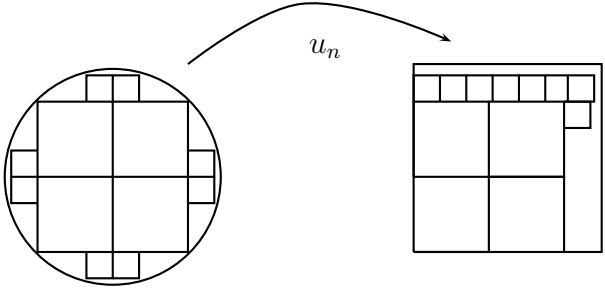
%%%%%%%%%%%%%%%%%%%%%%%%%%%%%%%%%%%%%%
In order to preserve some geometric and physical properties of  the shapes
under consideration, 
we then need  more regularity on the admissible reformations.
 %To fix ideas consider $X=  \Omega $ and $\mu =\mathcal L^N\res  \Omega$ for a  bounded connected open set $\Omega$. 
 We have the following 

\begin{lemma}
\label{on}
Let $x_0\in \Omega $, $u:X\rightarrow Y$.
If $u$ is differentiable at $x_0$ then
$$ r_u(x_0)=2\Rightarrow \nabla u(x_0)\in O(N),$$ where $O(N)$ denotes the set of orthogonal matrices.
\end{lemma} 
\begin{proof}
By \eqref{s1} we have $c_u(x_0)=e_u(x_0)=1$. Hence, for every $v\in \R^N$, taking $x=x_0+\delta v$ we get
$$c_u(x_0)= e_u(x_0)=1 \Rightarrow \frac{|\nabla u(x_0)\cdot v|}{|v|}=1 \Rightarrow \nabla u(x_0)\in O(N).$$
\end{proof}
By Liouville Theorem (see for instance \cite{ciarlet}) it follows that every 
$u\in C^1(X;Y) $ such that  $\mathscr{R}(u)=2$ is actually an isometry. 
There are  several generalizations of  Liouville Rigidity Theorem, 
however  (see \cite{ciarlet,camillo1, rigidity})  these  results are not directly applicable in our context since they generally require a constant sign for the Jacobian, as the condition $\nabla u(x) \in SO(N)$ for a.e. $x\in X$.
For instance, the map $u(x)= x$ if $x_1\geq 0$ and $u(x)=(-x_1,x_2, \ldots, x_N)$ if $x_1 \leq 0$ belongs to the Sobolev space $W^{1,2}(\Omega, \R^N)$, $\nabla u(x)\in O(N)$ for a.e. $x$,  but $u$ is not an isometry.\\
 Since  reformations have to preserve the volume, we have  the following result.
 
\begin{theorem}[Rigidity]\label{rigidity}
Let $U \subset \R^N$ be an open connected bounded set. Let $u:\overline U \rightarrow \R ^N $ be a continuous, locally Lipschitz, open map such that  
$\L^N( U)=\L^N( u(U))$ and satisfying the following conditions
%Let $u:\overline U \rightarrow \overline  V$ be a continuous  %(locally Lipschitz) map   satisfying the following conditions
 \begin{enumerate}
\item[(i)] $u(\partial U) \subset \partial (u(U))$%_\sharp \mu = \nu$
\item[(ii)] $u$ is a.e. differentiable and $\nabla u\in O(N)$ a.e. on $U$.
\end{enumerate} Then $u$ is an affine function.
\end{theorem}
\begin{proof}
By (ii) it follows that $u$ is locally a $1$-Lipschitz function (see   \cite[Proposition 3.4]{origami}). %Letting $N(y,u, U):=\mathcal H ^0(u^{-1}(y)\cap U)$ and $Ju:=\rm{det} \nabla u$, 
By the Area Formula and (ii) we infer
$$ \L^N(u(U))=\L^N( U)= \int_{ U} |Ju| \ dx = \int _{ \R^N} N(y,u,U)\ dy .$$ Therefore, $N(y,u,U)=1$ for a.e. $y\in  u(U)$.
Observe that \begin{equation}\label{deg} u(U)\subset \R^N \setminus \partial u(U)\subset \R^N \setminus u(\partial U).\end{equation} Then, $\forall x\in U$   ${\rm deg}(u(x),u,U)$ is well defined. Since $u$ is a.e. differentiable, for a.e. $x\in U$ it results (see Lemma 5.9 of  \cite{f-g})
$$ \left |{\rm deg}(u(x),u,U) \right | =  \left | {\rm sign}\ (Ju(x)) \right  | = 1.$$ On the other hand, since $u(U)$ is connected, by \eqref{deg}, $u(U)$  is contained in a connected component of  $\R^N \setminus u(\partial U)$. Therefore, the degree is constant on $u(U)$and so  the sign of the Jacobian $Ju$ is a.e. fixed. The conclusion follows by  Liouville  Theorem for Sobolev maps (see for instance \cite{ciarlet}).
\end{proof}
For a related  rigidity result involving local homeomorphisms see \cite{b-d}. For quasi-isometries over Banach spaces see \cite[Cor. 14.8]{nonlinear}.

\begin{remark}
Condition (i) holds of course for invertible maps $u$. If we are dealing with locally invertible maps, since   continuous and locally invertible maps are open maps, actually by (i) the  equality $ u(\partial U) = \partial (u(U))$ holds true. Moreover, if the map $u:\partial U \rightarrow \partial (u(U))$ is injective, then $u$ is globally invertible (see also \cite{invert-lang}).
Another classical condition for global invertibility holds for simply connected, or simply connectedly exhausted,  target $\overline {u(U)}$ (see \cite{ambrosetti-prodi, b-d}). 
Moreover, suppose to have a continuous, locally invertible,  surjective function $u:X\rightarrow Y$ such that $\nabla u(x)\in O(N)$  for a.e. $x$.
Then, $D=N(y,u,U)$ is constant (see \cite{ambrosetti-prodi}) and  by  Area  Formula we have
$$ D\L^N(Y)=  \int _Y N(y,u,U)dy = \int _X |Ju(x)|dx =\L^N(X).$$ 
Hence, if $\L^N(X)=\L^N(Y)$, it follows that $N(y,u,U)=1$ and hence $u$ is globally invertible.\\

%In ogni caso, si potrebbe ragionare localmente per le mappe con gradiente ortogonale tali che $u_\sharp \mu = \nu$. Se $u$  \'{e} localmenteinvertibile, si ha $$ |u(B)|= \int _Y \chi_{u(B)}(y)\ dy = \int _ X \chi_{u(B)}u(x)\ dx = \int _X \chi_B(x) \ dx =|B|,$$ essendo $x\in B\Leftrightarrow u(x)\in u(B)$ per locale ingettivit\'{a}. Ma allora la (i) vale localmente e quindi $u$ \'{e} localmete isometria. Quando  lo \'{e}dappertutto?\\Se $X, Y$ sono convessi, allora $u$  \'{e}  globalmente invertibile e $u,u^{-1}$ sono $1$-Lip. Pertanto $$ d(x,y)\leqd(u^{-1}(u(x)),u^{-1}(u(y)))\leq d(u(x),u(y))\leq d(x,y).$$ Questo potrebbe essere generalizzato per coppie di mappe $(u,v)$.
\end{remark}

\begin{theorem}\label{min1}
 Let $\mu \in \P(\Omega)$  and $\nu \in P(Y)$ so that  $\mu= \L^N\res \Omega$,  
 $\nu =\L^N \res Y$. Suppose that for $H,K\in L^1(X, \mu)$ provided by Definition  \ref{def-ref}   the inequality $ H(x)K(x)<2 $ is satisfied. 
Then the variational problem 

\begin{equation}
\label{p1} 
\min \{ \mathscr{R}(u) \;\vert\; u\in \Refo(\mu ;\nu)^{H,K}, u \mbox{ open }  \}
 \end{equation}
  admits solutions whenever $\{u\in \Refo(\mu ;\nu)^{H,K}, u \mbox{ open }  \}\neq \emptyset$.
\end{theorem}\begin{proof}
Since $\mu= \L^N\res \Omega$, we may assume that  $X=\Omega$. %Otherwise, we should assume  that $X\cap \overline B(x,r)$ is a convex, or a quasiconvex set. 
Let $(u_n)_{n\in\N}$ be a minimizing sequence. 
%By Morrey's inequality $$ |u(x)-u(y)|\leq C(N,p)|x-y|^{1-\frac{N}{p}} \| \nabla u \|_p $$, due to the condition $k\in L^p$ we see that the sequence $u_n$ is locally %equi-Holder.
%%%%%%%%%%%%%%%%%%%%%%%%%%%%%%%%%%%%%%%%%%%%%%%%%%%%%%%%%%%%%%%%%%%%%%
Given $x_0\in \Omega $,  by Definition \ref{def-ref} and Lemma \ref{lip-lemma} it follows that the sequence $(u_n)_{n\in\N}$ is locally equi-Lipschitz on $ \overline B(x_0,r)$. %, see Lemma \ref{lip-lemma} or \cite{point-lip}. 
%Therefore, the sequence $u_n$ is pointwise equicontinuous on $\Omega$. 
%Since $Y$ is compact, 
By  the Ascoli-Arzel\'{a} Theorem we  extract a subsequence  converging, uniformly on compact subsets of $\Omega$,   to a continuous map $u$.  
For this continuous limit map $u:\Omega \rightarrow \R^N$ it is  easily seen that  $u_\# \mu = \nu$.
%%%%%%%%%%%%%%%%%%%%%%%%%%%%%%%%%%%%%%%%%%%%%%
 %~~~~~~~~~~~~~~~~~~~~ 
 It remains to prove  %$u$ is open and %$v(\Omega ) =Y$ and that 
 $u\in \Refo(\mu;\nu)^{H,K}$, namely  $e_u(x)\leq K(x_0), c_u(x)\leq H(x_0)$ for every  $x\in X\cap \overline B(x_0,r)$. %,  and that $HK<2$.
 %%%%%%%%%%%%%%%%%%%%%%%%%%%%%%%%%%%%%%%%%%%%%%%%%%%%%%%%%%%%
 
 %%%%%%%%%%%%%%%%%%%%%%%%%%%%%%%%%%%%%%%%%%%%%%%%%%%%%%%
Since $X=\Omega$, by Lemma \ref{lip-lemma} we get the Lipschitz condition
$$|u_n(x_1)-u_n(x_2)|\leq K(x_0)|x_1-x_2| $$ for every $x_1,x_2 \in  \overline B(x_0,r)\subset \Omega $. Passing to the limit as $n\rightarrow + \infty$ and then as $r\rightarrow 0^+$,  we obtain $e_u(x)\leq K(x_0)$.

Observe that by Theorem~\ref{metric-inversion} the maps $u_n$ are locally invertible. Therefore, using Lemma~\ref{lip-lemma} and 
Lemma~\ref{reform} the inverse maps $u_n^{-1}$ are also locally  equi-Lipschitz. 
Moreover, by the theory of quasi-isometric mappings (see \cite[Theorem III]{john}, \cite{nonlinear}) the maps $u_n$ are equi-Lipschitz on the balls
$B(x_0, \frac{r}{HK})$. In this way 
   we find a common neighborhood $U_{x_0}:=B(x_0, \frac{r}{HK})\subset B(x_0,r)$  in which the functions $u_n$ are all simultaneously invertible (see also  \cite[Proposition 7]{inverse-sobolev},  \cite{inverses}). 
 It follows that $u$ is also locally invertible. Indeed, 
 suppose by contradiction to get   two distinct  sequences $(x_h^1)_{h\in \N}$, $(x_h^2)_{h\in \N}$ converging to $x_0$ such that $u(x_h^1)=u(x_h^2)$ $\forall h\in \N$. Let $\eps >0$ be fixed. By uniform convergence, we find a large integer $n$ such that $|u_n(x) -u(x)|<\eps$ for every $x\in U_{x_0}$.
 Now,  for a large $h$ we may assume $x^1_h, x_h^2 \in U_{x_0}$ and  we compute $$ |x_1^h-x_2^h|=|u_n^{-1}(u_n(x_h^1))-u_n^{-1}(u_n(x_h^2))|\leq H |u_n(x_h^1)-u_n(x_h^2)|\leq  $$ $$ H\left ( |u_n(x_h^1)-u(x_h^1)|+|u(x_h^1)-u(x_h^2)|+ |u(x_h^2)-u_n(x_h^2)| \right ) \leq 2H\eps ,$$ where $H$ is a common Lipschitz constant for $u_n^{-1}$.  By the arbitrariness of $\eps$ we get the contradiction $x_h^1=x_h^2$.
 %%%%%%%%%%%%%%%%%%%%%%%%%%%%%%%
 Observe that $u$ is open by the Domain Invariance Theorem. Hence  $u(\Omega)$ is actually an open set.  Let $x_1\in  \overline B(x_0,r)$ 
 and $y_1=u(x_1)$  be fixed. %and let $U_1$ be a common neighborhood  in which the functions $u_n$ are all simultaneously invertible. 
 Adding a constant, we may also suppose $u_n(x_1)=y_1$. 
 By using \cite[Th. II]{john} it results $B_1:=B(y_1,\frac{r}{H})\subset u_n(B(x_1,r))$, where the $u_n^{-1}$ are simultaneously defined.  %By using the bi-Lipschitz condition
 %$$ \frac{1}{H(x_0)}|x-x_1|\leq |u_n(x)-u_n(x_1)|\leq K(x_0)|x-x_1|$$
 %on $U_{x_0}$, 
 %For a common neighborood $x_1\in U_1$ we may consider the homeomorphisms $u_n : U_1 \rightarrow B(y_1,r_1)\subset u(B(x_0,r))$.
 
 %the homeomorphisms $u_n : Q_n \rightarrow B(y_1,r_1)$.

By Lemma \ref{reform} we get $e_{u_n^{-1}}(y) \leq H(x_0)$ for every $y\in B_1$.  
By Lemma~ \ref{lip-lemma} it follows
$$|u_n^{-1}(y)-u_n^{-1}(y_1)|\leq H(x_0)|y-y_1|,\:\:\:\forall y\in B_1.$$ 
 On the other hand,  the maps $u_n$ are bi-Lipschitz on $U_1=B(x_1, \frac{r_1}{HK})$. For the common neighborhood $U_1$ of $x_1$ we have
$$ |x-x_1|\leq H(x_0)|u_n(x)-u_n(x_1)|,\:\:\: \forall x\in U
_1.$$
Passing to the limit as $n\rightarrow + \infty$ and then  as $x\rightarrow x_1$  we get $c_u(x_1)\leq H(x_0)$. 
Hence $u\in \Refo(\mu ;\nu)^{H,K}$.\\ 
Fixed $\eps >0$, we find $\delta >0$ such that  $\int_{E} (H(x)+K(x))\ d\mu <\eps$ whenever $\L^N(E)<\delta$.
By using a Vitali covering,  we cover $\Omega$, up to a measurable set $E$ such that 
$\L^N(E)=\delta >0$,  by a finite number of disjoint neighborhoods $U_i$ on which $u_n \rightarrow u$ uniformly and  invertibility holds.
% according to Lemma~\ref{inver} and Lemma~\ref{local-inver}.\\ 
Since $u_\# \mu = \nu$ we  compute  
%\begin{eqnarray*}
\begin{equation*}\begin{split}&\mathscr{R}(u)
\leq \sum_{i=1}^l\left ( \int_{U_i}  \lip(u)(x) d\mu + \int_{U_i}  \lip(u^{-1})(u(x)) d\mu \right )+ \\
&\int_{E} (H(x)+K(x)) d\mu \leq \sum_{i=1}^l\left ( \int_{U_i}  \lip(u)(x) d\mu + \int_{u(U_i)}  \lip(u^{-1})(y)\ d\nu \right )+\eps. \end{split}\end{equation*}
%\end{eqnarray*}
By Lemma \ref{lsc}   we get 
\begin{eqnarray*}
\mathscr{R}(u)
&\leq& 
\sum_{i=1}^l \liminf_{n\rightarrow + \infty}\left (\int_{U_i}  \lip(u_n)(x) d\mu + \int_{u(U_i)}  \lip(u_n^{-1})(y)\ d\nu \right )+\eps\\
& \leq& 
 \liminf_{n\rightarrow + \infty}\sum_{i=1}^l\left (\int_{U_i}  \lip(u_n)(x) d\mu + \int_{u(U_i)}  \lip(u_n^{-1})(y)\ d\nu \right )+\eps\\
 & \leq& 
\liminf _{n\rightarrow + \infty}\left ( \int_\Omega e_{u_n}(x)\ d\mu + \int_{\Omega} c_{u_n}(x)\ d\mu \right ) +\eps = \liminf _{n\rightarrow + \infty} \mathscr{R}(u_n) + \eps .
\end{eqnarray*}
Letting $\eps \rightarrow 0^+$ we get the thesis. 
\end{proof}
 %~~~~~~~~~~~~~~~~~~~~
 \begin{remark}\label{remark-min}
 By using essentially the same tools employed in the proof of Theorem \ref{min1} and  according to Theorem \ref{invertibility1} and Theorem \ref{incompressible} one can prove existence results for the variational problems 
 $\min  \{ \mathscr{R}(u) \;\vert\; u\in A_i \},$ where  
 \begin{equation*}\begin{split}&A_1=\{u\in \Refo(\mu ;\nu)^{H,K}, e_u<\sqrt[N]{2} \},  A_2=\{u\in \Refo(\mu ;\nu)^{H,K}, u \mbox{ incompressible}  \}, \\ & A_3=\{u\in \Refo(\mu ;\nu)^{H,K}, u \mbox{ (locally) invertible}  \}.\end{split}\end{equation*}
 \end{remark}
 
\begin{remark}
If for reformation maps the surjection property $u(X)=Y$ is required, % for $u\in \Refo(\mu;\nu)^{H,K}$, 
we may argue as follows. 
To check that $u$ is onto, let us fix $y_0\in Y$. Observe that  $u_n^{-1}$ are locally equi-Lipschitz. Arguing as in the proof of Theorem in \ref{min1} for the sequence $u_n^{-1}$ we find a common neighborhood $B(y_0,r)$ such that 
$ u_n^{-1}$ %: B(y_0,r)\rightarrow U \subset \subset \Omega$$ 
are simultaneously homeomorphisms. Therefore, since $u_n(X)=Y$, we find a sequence $x_n\rightarrow x_0$ %\in U\subset \Omega$ 
such that $u_n(x_n)=y_0$. Then we have
\begin{eqnarray*}
|u(x_0)-y_0|
&\leq& 
|u(x_0)-u(x_n)|+|u(x_n)-u_n(x_n)|+|u_n(x_n)-y_0|\\
&\leq &
 |u(x_0)-u(x_n)|+\|u-u_n\|_\infty \rightarrow 0
 \end{eqnarray*}
as $n\rightarrow + \infty$.
\end{remark}
\begin{remark}
Observe that thanks to the compactness of $\Refo (\mu;\nu)^{H,K}$, no coercitivity conditions on  the energy functional $\mathscr{R}$ are needed (See \cite[Ch. II Section 9]{b-d} for related results in the setting of mappings with bounded distortion). In the case of $H,K \in L^p(X, \mu)$ the above minimization  result could be obtained  by using Rellich-Kondrakov compactness in Sobolev spaces and the l.s.c of the $p$-Dirichlet energy. 
%The point is that the above reasonings rely just on metric objects and can be generalized also in a metric framework by using the theory of Sobolev spaces over metric spaces (see Appendix~\ref{sobolev-metric-spaces}).
\end{remark}
%%%%%%%%%%%%%%%%%%%%%%%%%%%%%%%%%%%%%%%%%%%%%%%%
\begin{remark}\label{comp}
For $X$ compact, considering finite coverings, it turns out that $H, K \in L^\infty$. Therefore, in such a case we may consider $H,K$ as two universal constants.  However the proof of Theorem \ref{min1} works as well for the non-compact case. It would be interesting to develop an analogous theory under weaker requirement on the functions  $H,K$. 
For instance,  supposing $H,K\in L^p$ with $p>N$, by Morrey's inequality $$ |u(x)-u(y)|\leq C(N,p)|x-y|^{1-\frac{N}{p}} \| \nabla u \|_p $$  it follows  that the sequences $u_n, u_n^{-1}$ of the proof of Theorem \ref{min1} are locally equi-Holder. Hence we get existence of minimizers for example in the  set $A_3$ as in Remark \ref{remark-min}. 
%$\{u\in \Refo(\mu ;\nu)^{H,K}, u \mbox{ (locally) invertible}  \}$.
 To check that the set $A_3$  is closed one can also use the results of \cite{BIF, univalent, inverses}.

\end{remark}

%%%%%%%%%%%%%%%%%%%%%%%%%%%%%%%%%%%%%%%%%%%%
 Isometric measures are characterized by the following statement.
\begin{theorem}\label{main}
Let $\mu \in \P(\Omega)$ and   $\nu \in  \P(Y)$, so that 
$\mu= \L^N\res \Omega$,  $\nu =\L^N \res Y$, 
 for a given bounded set $Y$.
 Then, $\mathcal E(\mu, \nu)=2$ if and only if there exists an  isometry $u$ %(global for extension domains) %($u:\overline \Omega \rightarrow u(\overline \Omega)=\overline Y$?) 
 such that  $u_\# \mu = \nu$.
%~~~~~~~~~~~~~~~~~~~~
\end{theorem}
\begin{proof}
By Theorem \ref{min1} we get a minimizer $u:\Omega \rightarrow \R^N$ which belongs to $\Refo(\mu ;\nu)^{H,K}$. By Theorem \ref{invertibility1} and Remark \ref{invertibility1} it follows that $u$ is globally invertible. By Lemma \ref{on} and Theorem \ref{rigidity} it follows that $u$ is a local isometry, then (see for instance \cite[Th. 14.1]{nonlinear}),  $u$ is an isometric map. 

%By a Vitali covering  argument it follows that actually $u\in W^{1,\infty}(\Omega)$. Since $Y$ is   bounded, if $\Omega $ is an extension domain, it follows that (see for instance \cite{a-t}) $u$ is a global Lipschitz function on $\Omega$. Therefore $u$ is uniquely extended to $\overline \Omega$.By invertibility it follows that $u(\partial \Omega)\subset \partial u(\Omega )$. Therefore, by Theorem \ref{rigidity} we get that $u$ is an isometry. 
\end{proof}

%~~~~~~~~~~~~~~~~~~~~

%%%%%%%%%%%%%%%%%%%%%%%%%%%%%%%%%%%%%%%%%%%%%%%%%%%%%%%%%%%%%%%%%
By Theorem~\ref{main} we  have that  by reforming a flat configuration 
$\mu$ in a corrugated one $\nu$ it results 
$\mathcal E(\mu, \nu)>2$. This last fact gives an alternative proof of the so-called Grinfeld instability (see \cite{FGM}),  indeed,  by the changing of the geometry, any possible reformation must expand or contract some portion of the body.

\section{Generalized reformations}
The notion of reformation introduced in the previous section has some restrictions, indeed  it is easy to exhibit examples, like the one in Figure $4.1$, %\ref{piano}, 
in which every reformation map has a large cost while  allowing fractures of the body leads to map the initial measure  by using local isometries. 

%%%%%%%%%%%%%%%%%%%%%%%%%%%%%%%%%%%%%%%%%%%
\begin{figure}[htbp]\label{piano}	\centering		\includegraphics[scale=0.8]{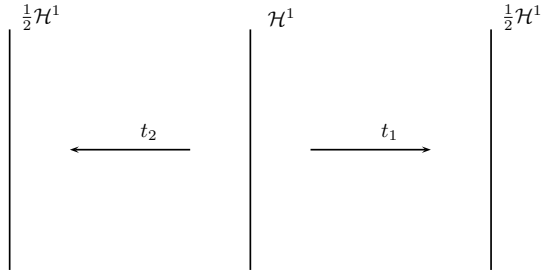}	\caption{An isometric fractured reformation.}\end{figure}
%%%%%%%%%%%%%%%%%%%%%%%%%%%%%%%%%%%%%%%%%%%%%%%%%%%%%
%(segmento unitario da trasportare metï¿½ da una parte e metï¿½ dall'altra) 

 %~~~~~~~~~~~~~~~~~~~~
Here we  introduce a notion of \emph{generalized reformation}. Our approach relies on measure theoretic tools mostly developed 
in the field of optimal 
 mass transportation  (see \cite{a1, v2}) where maps satisfying  $u_\#\mu=\nu$ are called transport maps. 
 %The notion of reformation map corresponds to the notion of the so-called transport map, i.e. $u: X\rightarrow Y$ such that $u_\#\mu=\nu$. 
 A natural generalization of the transport map (reformation map) is given by the notion of transport plan. A transport plan between two probability measures  $\mu\in\P(X)$ and $\nu\in \P(Y)$ 
 is a measure $\gamma \in \P(X\times Y)$ such that 
 $\pi ^1_ \# \gamma = \mu$, $\pi ^2_\# \gamma = \nu$, where $\pi ^i $, $i=1,2$  denote the projections of $X\times Y$ on its factors. 
 A transport map
 $u$ induces the transport plan $\gamma_u:= (I\times u)_ \# \mu$, where $I$ is the identity map of $X$. Observe that the set of transport plans
  with marginals $\mu$ and $\nu$,   denoted by $\Pi(\mu,\nu)$, 
   is never empty since  it always contains the transport plan $\mu \otimes \nu$.\\ 
 We shall call \textit{generalized reformation}, or \textit{reformation plan},  of $\mu$ into $\nu$ any transport plan $\gamma$  with marginals $\mu$ and $\nu$.\\ 
 Let us recall some known results which 
     will  play a crucial role in the following   (we refer to \cite{AFP, a1}). 

%~~~~~~~~~~~~~~~~~~~~

\begin{definition}
\label{borel} 
Let $\M(Y)$ be the space of Radon measures on $Y$. 
A map $\lambda :X\rightarrow \M (Y)$ is said to be Borel if for any open set 
$B\subset Y$ the function $x\in X \mapsto \lambda_x (B)$ is a real valued Borel map.
Equivalently, $x\mapsto \lambda_x$ is a Borel map if for any Borel and bounded map $ \varphi :
X\times Y \rightarrow \R$ it results that the map $$ x \in X
\mapsto \int_Y \varphi (x,y) d\lambda _x $$ is Borel.
\end{definition}
\begin{theorem}[Disintegration theorem]
\label{disintegration} Let $\gamma  \in \P (X\times Y ) $ be given and let
$ \pi^1 : X\times Y  \rightarrow X$ be the first projection map of $X\times Y$, we set $\mu = (\pi^1)_ \# \gamma$.  Then for $\mu-a.e.\ x \in X $ there exists $\nu _x \in \P
(Y)$   such that
\begin{itemize}
\item[\rm (i)] the map $x \mapsto \nu _x$ is Borel,\\
\item[\rm (ii)]$ \forall \varphi \in \C _b (X\times Y ) : \, \int _{X\times Y} \varphi (x,y) d\gamma
= \int _X \left ( \int _Y \varphi (x,y) d\nu _x(y)
\right )d\mu (x) .$
\end{itemize}
Moreover the measures $\nu _x$ are uniquely determined up to a negligible set with respect to  $\mu$. 
\end{theorem}

 %~~~~~~~~~~~~~~~~~~~~

Let $\gamma\in\Pi(\mu,\nu)$,
as usual we will  write $ \gamma =  \nu _x \otimes \mu $, 
 assuming that $ \nu _x $
satisfy the condition (i) and (ii) of Theorem~\ref{disintegration}. Obviously the transport plan $\mu \otimes \nu$ corresponds to the constant map $x\mapsto\nu_x=\nu$. 
Let $u: X\rightarrow Y$, 
observe that for the  transport plan $\gamma _u:= (I\times u)_ \# \mu$, the  Disintegration Theorem yields $\gamma _u=  \delta _{u(x)}\otimes \mu $.

\begin{remark}
{\rm Let  $X\subset \R ^N$,  we recall that the  barycenter of a measure $\mu\in \P(X)$ is given by 
$$\beta (\mu )= \int _{X}x\ d\mu .$$ 
If $\gamma =  \nu _x \otimes \mu$, then, by Theorem \ref{disintegration} the map $x\mapsto \beta (\nu _x)$ is measurable. It is possible to define a generalized pointwise expansion and  compression  energy through the pointwise Lipschitz constant of the map 
$x\mapsto \varphi (x):=\beta (\nu _x)$. 
Observe that for a transport map $u$, since $\beta (\delta _x)=x$,  we have $$r_{\varphi  }(x_0)=r_u(x_0).$$
However, it may happen that the map $\varphi $ is an isometry although the target are  quite far from being \emph{isometric} as described  in Figure $4.2$ %\ref{baricenter-fig}
.}
\end{remark}
\begin{figure}[htbp]\label{baricenter-fig}	\centering		\includegraphics[scale=0.8]{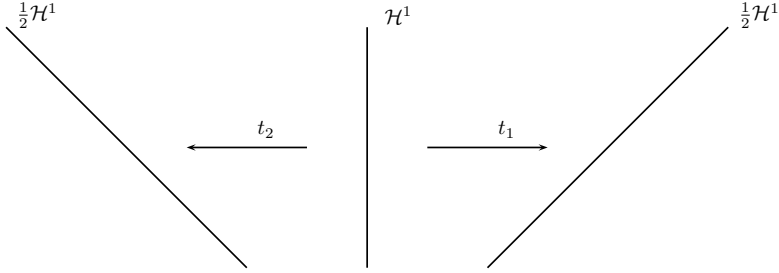}	\caption{A barycenter isometric reformation.}\end{figure}

%%%%%%%%%%%%%%%%%%%%%%%%%%%%%%%%%%%%%%%%%%%%%%%%%%%%%%%%%%%%%%%%%%%%%%%%%%%%%%%%%%%%%%%%%%%%%%%%%%%%%%%%%%%%%%%5
%~~~~~~~~~~~~~~~~~~~~

In the sequel  we will introduce the notion of
generalized pointwise compression and expansion energy through the 
notion of  $1$-Wasserstein distance of measures.

\begin{definition}
\label{wd} Let $\mu, \nu \in \P(X)$, the $1$-Wasserstein distance between $\mu$ and $\nu$ is defined  by
\begin{equation}
W(\mu,\nu )= \inf_{\gamma\in\Pi(\mu,\nu)} \int_X d(x,y) \
 d\gamma(x,y). 
\end{equation}
\end{definition}
Let us recall that by Kantorovich duality (see \cite{a1,granieri, v2}) the $1$-Wasserstein distance between $\mu$ and $\nu$ can be expressed as follows
\begin{equation}
\label{wd1}
W(\mu,\nu )=\sup\left\{\int_X\varphi\;d(\mu-\nu)\;\vert\; \varphi\in\lip_1(X)\right\}.
\end{equation}

\begin{lemma}
\label{convex}
The balls of $(\P(Y),W)$ are $1$-convex. 
\end{lemma}
\begin{proof}
Let $\mu\in\P(Y)$, $r>0$ be fixed, we consider   $\nu_1, \nu _2 \in B:=B(\mu,r)\subset \P(Y)$.
For every $t\in[0,1]$, let  $\nu_t:=t\nu_1+(1-t)\nu_2$. 
Then, by considering \eqref{wd1}, for any fixed  $\varphi\in \lip_1(Y)$, we compute
\begin{eqnarray*}
  \int_Y \varphi \ d(\nu_t-\mu)
  &=&
   t \int_Y \varphi \ d(\nu_1-\mu)+(1-t)\int_Y\varphi\ d(\nu_2-\mu)\\
  & \leq&
   tW(\nu_1,\mu)+(1-t)W(\nu_2,\mu)\leq r %\\   &\leq &    r.
   \end{eqnarray*} 
Passing to the supremum with respect to $\varphi\in\lip_1(Y)$ we get
$W(\nu_t,\mu)\leq r$, hence $\nu_t\in B$ $\forall t\in[0,1]$. Observing that $W(\nu_{t+h}, \nu_t)=hW(\nu_1,\nu_2 )$ it follows that the length of the curve $\nu_t$ (see Appendix A) amounts to
$ l(\nu_t)=\int_0^1 | \dot   \nu_t| dt = W(\nu_1,\nu_2 )$.
\end{proof}
As stated in Section $1$, over the metric space $(\P(Y), W)$ the above Lemma allows to derive local Lipschitz conditions  from just pointwise Lipschitz bounds  (see also \cite{point-lip}).
Let  $\gamma =  \nu_x\otimes \mu $,  the function 
\begin{equation}
\label{f}
f:X\rightarrow (\P(Y), W),\:\:\:\: f(x)=\nu_x. 
\end{equation}
will be called \textit{disintegration map}.
Let us  introduce the notion of generalized compression and expansion energy in terms of the disintegration map $f$.
 %~~~~~~~~~~~~~~~~~~~~
\begin{definition}[Generalized expansion and  compression  energy]
\label{g-c-e}
For any 
 reformation plan  $\gamma =  \nu_x\otimes \mu $ of $\mu$ into $\nu$ we define the pointwise   expansion  energy
 \begin{equation}
\label{gesp}
  e_\gamma(x_0):= \limsup_{x\rightarrow x_0} \frac{W (\nu _x, \nu _{x_0})}{|x-x_0|},
  \end{equation}
   and  the pointwise compression energy 
    \begin{equation}
\label{gcomp}
        c_\gamma(x_0)= \limsup_{x\rightarrow x_0} \frac{|x-x_0|}{W(\nu _x,  \nu_{x_0})}.
\end{equation}        
          \end{definition}
\noindent  By using \eqref{f} we can state
\begin{equation}
\label{gamma-f}
  e_\gamma(x)=e_f(x),\:\: c_\gamma (x)=c_f(x). 
\end{equation}
The pointwise reformation energy is then defined by
$$ r_\gamma (x_0)=e_\gamma(x_0)+c_\gamma(x_0).$$
%~~~~~~~~~~~~~~~~~~~~

 \begin{remark}
 {\rm Notice  that,  since $W(\delta _x, \delta _y)=|x-y|$, if $\gamma$ is a reformation plan induced by a map $u:X\rightarrow Y$, say   $\gamma_u=(I\times u)_\# \mu$ and $f_u$ is the disintegration map of $\gamma$, then it results 
 $$r_\gamma (x_0)=r_{f_u}(x_0)=r_u(x_0).$$ }
 \end{remark}

\begin{definition}
\label{constraint-plans}
Given   $H,K:X\rightarrow  ]0,+\infty [$,  $H,K \in L^1(X, \mu)$ and a fixed covering $\mathcal A$ of $X$ made by balls, % $\subset \Cup_{x\in X, r>0}B(x,r)$ be flocally bounded, i.e
 we define the set $  \GRefo (\mu;\nu)^{H,K}\subset \Pi(\mu,\nu)$ as the subset of reformation plans $\gamma$ of $\mu$ into $\nu$   such that
 \begin{equation}\label{k} \forall x_0 \in X : \ \exists\; B(x_0,r)\in \mathcal A  \ s.t.\  e_\gamma(x)\leq K(x_0),\; c_\gamma (x)\leq H(x_0)\end{equation}
  $\forall x\in \Omega \cap \overline B(x_0,r)$. 
 \end{definition}
%~~~~~~~~~~~~~~~~~~~~

\begin{remark}
{\rm  
By \eqref{gesp}-\eqref{gamma-f} the role played by the disintegration map is clear, 
hence one is led to argue as in the previous section  trying to establish the analogous of 
Theorem~\ref{invertibility1} in the case of disintegration maps.
Unfortunately in the general  case of metric spaces some tools as degree theory are not available. Therefore, it is not clear if local invertibility follows by \eqref{k}.

Nevertheless, by restricting the analysis to the case of \textit{small reformations}, i.e. satisfying $HK\leq \mu_0$, for enough small constant $\mu_0$, 
it is possible to prove some  global invertibility results suitable to the present case.
For instance, assuming   that $\Omega$ is a ball and that $f$ is a local homeomorphism, then there exists a constant $\mu_0$ such that $HK<\mu_0$ implies $f$ globally invertible (see \cite{john, nonlinear} and \cite{gevirtz} for other classes of domains $\Omega$). 
}
\end{remark}
\begin{definition} Let us define  %the set of \emph{invertible} reformation plans between $\mu$ and $\nu$ as follows
\begin{equation}\label{smarr_ref_plans} \GRefo_0(\mu,\nu)^{H,K}=\{\gamma\in \GRefo(\mu,\nu)^{H,K}\:\vert\:\gamma =f(x)\otimes \mu, f:\Omega \rightarrow \P(Y) \mbox{ invertible }\}. \end{equation}  \end{definition}  

 %%%%%%%%%%%%%%%%%%%%%%%%%%%%%%%%%

%\begin{figure}	\centering		\includegraphics{baricenter.eps}	\label{fig:baricenter}\end{figure}
\section{Finding reformation plans}

 %~~~~~~~~~~~~~~~~~~~~
In the following examples we show  that it is possible to compare shapes with regular disintegration maps despite no regular transport map does  exist. 
%\vspace{1cm}
%%%%%%%%%%%%%%%%%%%%%%%%%%%%%%%%%%%%%%%%%%%%%%%%%%%%%%%
%\begin{figure}[htbp]\label{disco}
\begin{figure}[htbp]
	\centering
		\includegraphics[scale=0.8]{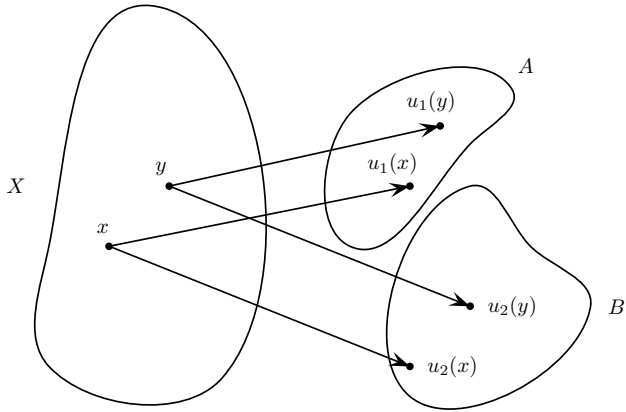}
	\caption{A disconnected target reformation}
	\label{disco}
\end{figure}

%%%%%%%%%%%%%%%%%%%%%%%%%%%%%%%%%%%%%%%%%%%
%\begin{figure}[htbp]\label{disco}

%\begin{pspicture}(-2,-2)(8,2.8)\psgrid(-2,-2)(8,2.8)\psdots(0,1)(1,2)\psdots(5,2)(5.5,3)\psdots(5,-1)(6,0)\psccurve(-1,-1)(-1,1)(0.7,5)(1.87,-1.2\psccurve(4,1)(4,3.2)(6.7,3.6)(6,2.7)\psccurve(5,-1.5)(6,2)(7,1)(8,0)\psline[arrowsize=7pt]{->}(0,1)(5,2)\psline[arrowsize=7pt]{->}(0,1)(5,-1)\psline[arrowsize=7pt]{->}(1,2)(5.5,3)\psline[arrowsize=7pt]{->}(1,2)(6,0)\uput*[0](-0.5,1.3){ $x$ }\uput*[0](0.48,2.3){ $y$ }\uput*[0](4,2.365){ $u_1(x)$ }\uput*[0](4.65,3.4){ $u_1(y)$ }\uput*[0](5,-1){ $u_2(x)$ }\uput*[0](6,0){ $u_2(y)$ }\uput*[0](-2,2){ $X$ }\uput*[0](6.5,4){ $A$ }\uput*[0](8,0){ $B$ }\end{pspicture}\caption{A disconnected target reformation.}\end{figure}
%%%%%%%%%%%%%%%%%%%%%%%%%%%%%%%%%%%%%%%%%%%%%%%%%%%%%%%%%%%%%%%%%%%%%%%%%
\begin{example}\label{disjoint}Consider a regular domain $X\subset \R^N$  splitted into $Y=A\cup B$ for two disjoint regular domains $A,B\subset\R^N$ in such a way $1=\L^N(X)=\L^N(A)+\L^N(B)$. We find (see  \cite{ bilip, wolan}) two diffeomorphisms
$u_1 : X\rightarrow A$, $u_2:X\rightarrow B$ so that 
$|Ju_1|=\L^N(A)$, $ |Ju_2|=\L^N(B)$.\\
Diffeomorphisms with constant Jacobian can be constructed by using  the results of \cite{d-m}. Indeed, let $\varphi : \Omega \rightarrow \Omega_1$ be a diffeomorphism. Assume for instance  $J\varphi(x) >0$ $\forall x\in \Omega$ and let   $f(x)=\frac{\L^N(\Omega)}{\L^N(\Omega_1)}J\varphi(x)$. 
Then
$$\int_\Omega f(x) \ dx =\frac{\L^N(\Omega)}{\L^N(\Omega_1)}\int_\Omega J\varphi (x) \ dx =\L^N(\Omega).$$ By the results of \cite{d-m}, there exists a diffeomorphism $u:\Omega \rightarrow \Omega$ such that $Ju=f$. Setting $\psi=\varphi \circ u^{-1} :\Omega \rightarrow \Omega_1$ it follows that 
$J\psi=\frac{\L^N(\Omega)}{\L^N(\Omega_1)}$.
%%%%%%%%%%%%%%%%%%%%%%%%%%%%%%%%%%%%%%%

%\vspace{2cm}
%%%%%%%%%%%%%%%%%%%%%%%%%%%%%%%%%%%%%%%%%%%%%%%%%%

 Let  $\nu_x=\L^N(A)\delta_{u_1(x)}+\L^N(B)\delta _{ u_2(x)}$, 
then  the reformation plan $\gamma :=  \nu_x\otimes \mu  $ has $\mu =\L^N\res X$ and $\nu=\L^N\res Y$ as marginals. We claim that the function $f(x)=\nu_x$ is, at least locally,  bi-Lipschitz. Indeed it results
$$ W(\nu_x, \nu _{x_0})=\L^N(A)|u_1(x)-u_1(x_0)|+\L^N(B) |u_2(x)-u_2(x_0)|.$$ Since $u_1, u_2$ are diffeomorphisms, we find constants $K_{1,2}, H_{1,2}, H,K$ such that
\begin{eqnarray*}
\frac{1}{H}|x-x_0|
&\leq& \frac{\L^N(A)}{H_1} |x-x_0|+\frac{\L^N(B)}{H_2}|x-x_0|\\
& \leq &
\L^N(A)|u_1(x)-u_1(x_0)|+\L^N(B)|u_2(x)-u_2(x_0)|\\
&=&
W(\nu_x, \nu _{x_0}),
\end{eqnarray*}

\begin{eqnarray*} 
W(\nu_x, \nu _{x_0})
&=&
\L^N(A)|u_1(x)-u_1(x_0)|+|\L^N(B) |u_2(x)-u_2(x_0)|\\
&\leq &  
\L^N(A)K_1|x-x_0|+\L^N(B)K_2|x-x_0|\\
&\leq & 
K|x-x_0|.
\end{eqnarray*}
%%%%%%%%%%%%%%%%%%%%%%%%%%%%%%%%%%%%%

\begin{remark} 
The above construction is possible also for a class of star-shaped domains as in \cite[Theorem 5.4]{fonseca-parry} by considering bi-Lipschitz maps in place  of diffeomorphisms.
\end{remark}
\end{example}

Moreover, generalized reformation maps are useful to compare \emph{near-isometric} shapes.
%%%%%%%%%%%%%%%%%%%%%%%%%%%%%%%%%%

%%%%%%%%%%%%%%%%%%%%%%%%%%%%%%%%%%%%%%%%%%%%%%%%%
\begin{example}\label{bending}
Consider a rectangle $R$ and a bended one with the bended size of $\frac{1}{n}$ (see Figure $0.1$). 
Consider the maps
$$u_1(x)=\left (1-\frac{1}{n}\right )\left ( Ax+a\right ) , \quad u_2(x)=\frac{1}{n}\left ( Bx+b\right ) $$ for orthogonal  matrices $A,B$ and then the reformation plan 
$$ \gamma= \left ( \left (1-\frac{1}{n}\right )\delta_{u_1(x)}+\frac{1}{n}\delta_{u_2(x)}\right ) \otimes \mu ,$$
where $\mu = \L^N\res R$.
We compute
$$ W(\nu_x,\nu_{x_0})= \left (1-\frac{1}{n}\right )|u_1(x)-u_1(x_0)|+\frac{1}{n}|u_2(x)-u_2(x_0)|=
\left (\left ((1-\frac{1}{n}\right )^2+\frac{1}{n^2}\right )|x-x_0|.$$
Therefore the function $f(x)=\nu _x $ is, at least locally, bi-Lipschitz and 
$$ e_\gamma (x_0)= \left (1-\frac{1}{n}\right )^2+\frac{1}{n^2}\rightarrow 1 $$ as $n\rightarrow + \infty$, while 
$ c_ \gamma (x_0)=\frac{1}{e_\gamma (x_0)}$. 
\end{example}

%~~~~~~~~~~~~~~~~~~~~
%%%%%%%%%%%%%%%%%%%%%%%%%%%%%%%%%%%%%%%%%%%%%
\begin{figure}[htbp]\label{bari}\begin{pspicture}(-5.5, -3.5)(8.5, 3.5)\psaxes[labels=none, linecolor=gray]{->}(0,0)(-3.5,-2.5)(5.5,3)
\psplot{0}{1}{x 1 add}\psplot{0}{1}{x -1 mul -1 add}\uput*[0](0,0.4){$2\sqrt 2$}\psline(0,0)(1,0)\uput*[0](1,2){$y=x+1$}\uput*[0](1,-2){$y=-x-1$}\end{pspicture}\caption{An horizontal segment, with mass $2\sqrt2$, splitted into two different ones. }\end{figure}
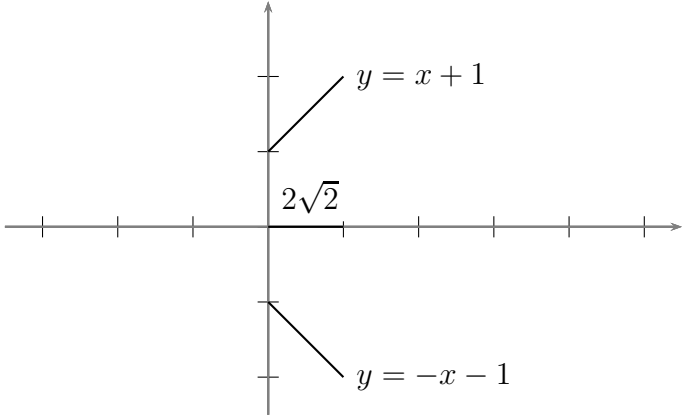
%%%%%%%%%%%%%%%%%%%%%%%%%%%%%%%%%%%%%%%%%%%%%%%%%%%%%%%%%%%%%%%%%%%%%%%%%%
\begin{example}Consider the situation displayed  in Figure $5.3$. %\ref{bari}
Defining $\nu _x=\frac{1}{2}\left ( \delta _{x+1}+\delta_{-x-1}\right )$,  we find $W(\nu_x,\nu_{x_0})=\sqrt 2 |x-x_0|$. Hence
$r_\gamma = e_\gamma+c_\gamma =\sqrt 2 + \frac{1}{\sqrt 2}$. 
\end{example}
\begin{example}
Let  $X\subset\R^N$ be a measurable set with $\L^N(\partial X)=0$.  We find an increasing sequence of polyhedral sets $X_n$ such that
 $X=\bigcup _{n\geq 1} X_n$ up to a negligible set. Let $Y\subset\R^N$ 
 be the unitary  cube,  $\L^N(X)=\L^N(Y)$ and let $Y_n\subset Y$ be a rectangle such that $\L^N(Y_n)=\L^N(X_n)$ $\forall n\in\N$.
 Let $\mu=\L^N\res X$, $\nu =\L^N\res Y$.\\
  We find a sequence $(u_n)_{n\in\N}$ so that $\forall n\in\N$ $u_n:X_n\rightarrow Y_n$ is a bi-Lipschitz map with $Ju_n=1$. 
  The volume constraint implies that $K_n:= \lip(u_n)\leq K$, $H_n:= \lip(u^{-1}_n)\leq H$.   
 In particular, for every $x,y\in X_n$ we have 
 $$ \frac{1}{H}|x-y|\leq |u_n(x)-u_n(y)|\leq K|x-y|.$$ 
 By Lipschitz extension, we may consider $u_n$ as defined on the whole $X$. By  Ascoli-Arzel\'{a} Theorem we find $u_n \rightarrow u$ uniformly. It follows that 
 $$ \frac{1}{H}|x-y|\leq |u(x)-u(y)|\leq K|x-y|,$$ up to a zero measure set. 
 Moreover 
 \begin{eqnarray*}
   \int_X f(u(x)) \ dx 
   &= & \lim_{n\rightarrow + \infty}\int _X  f(u_n(x))\ dx \\
 &=&
 \lim_{n\rightarrow + \infty}\left ( \int _{X_n}  f(u_n(x))\ dx +
 \int _{X\setminus X_n}  f(u_n(x))\ dx \right )\\
 & = &
 \lim_{n\rightarrow + \infty} \int _{Y_n}  f(y)\ dy= \int_Y f(y)\ dy .
 \end{eqnarray*} 
 Hence, $u_\# \mu =\nu$. 
\end{example}

\section{Variational problems on generalized reformations}

 %~~~~~~~~~~~~~~~~~~~~
The notion of generalized reformation involves the Lipschitz pointwise constant of maps in a metric space framework. For the associated integral energies it is natural to consider some notion of Sobolev spaces in a metric setting. There exist different notions of  metric Sobolev spaces which coincide provided some  mild assumptions such  as a \emph{doubling condition}, a Poincarè inequality and a power of integrability $1<p<+\infty$ are satisfied. We refer the reader to the Appendix B and the references therein for further informations. In particular, the  requirement on the power $p>1$ will be important to state a general existence result (see Theorem \ref{existence-plan} below)  for the variational problem related to generalized reformations. 
Actually, these kind of assumptions seem to form a natural context
 to work with in the setting of metric analysis. Therefore, along all  this section we will assume 
\begin{equation}
\label{domain}
\begin{array}{l}
X=\overline{\Omega}\subset \R^N \hbox{ compact and satisfying \eqref{d-i} and \eqref{p-i}},\\
\\
Y\subset\R^N \hbox{ compact}.
\end{array}
\end{equation}

\begin{definition}
Let $\gamma\in\Pi(\mu,\nu)$. %$\gamma= f(x)\otimes \mu$. 
%For every $p>1$ 
We define the reformation energy of $\gamma$ as follows
\begin{equation}
\label{generalized_energy}
{\mathscr R}(\gamma)=\int_X(c_\gamma+e_\gamma)\ d\mu.
\end{equation}
\end{definition}

\begin{remark}
With  abuse of notation we are  using the same symbol ${\mathscr R}$
to denote the reformation energy functional defined on the space of reformation maps and the analogous defined on the space of reformation plans. Since in the paper it always appear with its argument specified, there is no risk of confusion.
\end{remark}
%~~~~~~~~~~~~~~~~~~~~
\begin{theorem}
\label{iso-plan}
Let $\gamma =f(x)\otimes \mu \in\GRefo (\mu;\nu)^{H,K}$ be such that  $\mathscr R(\gamma) =2$, $\mu$ absolutely continuous with respect to the Lebesgue measure. Then there exists an open dense subset of $\Omega$ on which  the disintegration map $f$ is a local isometry (with respect to the Wasserstein distance).
\end{theorem}
\begin{proof}
First observe that since $\Omega $ is quasiconvex (see for instance \cite[Lemma 6.1]{metric-sobolev}), then $f$ is a Lipschitz function. 
We have $e_\gamma= c_\gamma =1$ a.e. By \cite[Prop. 1.1, Sec. 3]{open-metric}, there exists  an open dense subset $U\subset \Omega $ on which $f$ is locally bi-Lipschitz. Therefore,  consider  a bi-Lipschitz map $f:B\rightarrow \P(Y)$ for an open ball $B\subset U$. For $x_1,x_2 \in B$, by using Fubini Theorem, we find a curve $\eta $ connecting $x_1,x_2$ as in \cite[Prop. 3.4]{origami} in such a way for a.e. $t$ it results $e_\gamma(\eta (t))=1$
and $l(\eta)\leq |x_1-x_2|+\eps$.   Since $f$ is Lipschitz, the curve 
$\rho :[0,1]\rightarrow \left (\P(Y), W \right )$, defined by $\rho_t=f(\eta(t))$ is Lipschitz too. Hence, it admits a tangent vector $v$ (see Theorem \ref{tcont}).
Fixed $u\in \lip_1(Y)$, by standard approximation argument we may suppose that $u\in \C^1$. Therefore, by using the continuity equation \eqref{eqcont} we compute
$$ \int_Y u \ d\left (f(x_1)-f(x_2) \right ) = \int _Y u \ d(\rho_1-\rho_0)= \int_0^1 \frac{d}{dt} \left ( \int_Y u\ d\rho_t \right ) dt =$$
$$= \int_0^1 \int_Y \langle du,v\rangle d\rho_t\ dt \leq \int _0^1 \int_Y|v|d\rho_t \ dt = \int_0^1 |\dot \rho |(t)\ dt \leq \int_0^1 e_\gamma(\eta _t)|\dot \eta | dt \leq $$ $$\leq l(\eta ) \leq |x_1-x_2|+\eps .$$ Taking the supremum with respect to $u$ and letting $\eps \rightarrow 0^+$ we get the $1$-Lipschitz property
$$W(f(x_1),f(x_2))\leq |x_1-x_2|.$$ 
To get the opposite inequality, we argue as follows.
Set $\rho_0=f(x_1), \rho _1=f(x_2)$, 
let us consider a geodesic $\rho _t :[0,1]\rightarrow \P(Y)$ between $\rho _0$ and $\rho _1$, i.e. $l(\rho)=W(f(x_1),f(x_2))$. 
Since $f$ is bi-Lipschitz, there exists an injective Lipschitz curve 
$\gamma :[0,1]\rightarrow B$ connecting $x_1, x_2$ such that $\rho_t=f(\gamma(t))$. 
Again by using a Fubini type argument, we find a sequence of Lipschitz injective curves  $(\gamma_n)_{n\in \N}$ so that  $\gamma_n \rightarrow \gamma$ uniformly and  $\lip(f^{-1})(f(\gamma_n (t)))=1 $ for a.e. $t\in [0,1]$. Therefore, we get $\sigma _n =f(\gamma_n) \rightarrow \rho $ uniformly in $( \P(Y), W)$. By the upper semicontinuity of the Hausdorff measure along the sequence $\sigma _n$ (see for instance  \cite[Lemma 4.1]{urban}), recalling that for injective curves it results $l(\sigma)=\H^1(\sigma ([0,1]))$ (see \cite{a-t}),  fixed $\eps >0$,  we find a  Lipschitz curve $\sigma $ connecting $\rho_0$ and $\rho _1$ such that $\lip(f^{-1})(\sigma (t))=1 $ for a.e. $t\in [0,1]$ and $l(\sigma)\leq W(f(x_1),f(x_2))+\eps$. Finally, we compute
$$|x_1-x_2|=|f^{-1}(\sigma(0))-f^{-1}(\sigma (1))|=\left |\int_0^1\frac{d}{dt}f^{-1}(\sigma(t)) \ dt \right |\leq$$ $$\leq \int_0^1 |\dot \sigma |_W(t) \ dt =l(\sigma) \leq  W(f(x_1),f(x_2))+\eps .$$
Letting $\eps \rightarrow 0^+$ we get the thesis.
\end{proof}
Theorem \ref{iso-plan} should be compared with Theorem 3.17. The main restriction is on invertibility which is just on an open dense subset. We may say that this open set is of full measure, actually coinciding with the whole space, just for the case of small reformations as done in Theorem \ref{min2} below. There are different restrictions in doing so for the general case. A first matter relies in characterizing the set where a map is locally invertible on a metric setting. A second one relies on the fact that the integral functional $  \mathscr R$ gives a.e. informations, while invertibility requires global conditions. Therefore the matter is on passing from a.e. conditions to everywhere ones. In the results concerning reformation maps, this difficulty was overcome by using degree theory in $\R^N$. Therefore, something similar to degree theory over metric spaces should be needed in order to handle with this kind of questions.

Let us introduce the notation
\begin{equation}
\mathcal E_G(\mu,\nu)=\inf\{ \mathscr R(\gamma )\:\vert\:\gamma\in \GRefo_0 (\mu;\nu)^{H,K}\}.
\end{equation}
Concerning symmetry properties of the above generalized reformation energy, the same reasonings made for transport maps,  compare with  Definition \ref{elastic}, hold as well. We remark here that this time the question of symmetry is not just a question on invertibility. For instance, the transport plan $\gamma =f\otimes \mu$ between $\mu$ and $\nu$,  considered in Figure $4.1$ is isometric, i.e. $W(f(x),f(x_0))=|x-x_0|$. However, reversing the target measures we see that the transport plan between $\nu$ and $\mu$ is just locally isometric and no transport plan $g\otimes \nu$ between $\nu$ and $\mu$ is isometric. The fact is that the corresponding disintegration maps are of the form 

$$g:Y\rightarrow \P(X).$$
Therefore, symmetry questions are quite involved and here we do not further  consider them.

We state the following characterization of the lowest possible value of the generalized reformation energy. 

\begin{theorem}
\label{min2}If 
$\mathcal E_G(\mu,\nu)=2$, with $\mu=\L^N\res \Omega $,  %absolutely continuous with respect to the Lebesgue measure,  
then the infimum is attained at  a local isometric reformation plan. \end{theorem}
\begin{proof}
Since $\mu=\L^N\res \Omega $, we may assume $X= \Omega$. Let $\gamma _n$ be a minimizing sequence. By compactness of $\P(X\times Y)$,
 by passing to a subsequence, we may assume that $\gamma _n \stackrel{*}{\rightharpoonup}\gamma$. It follows that $\gamma$ is also a transport plan between $\mu$ and $\nu$. By disintegration, we also assume that
$\gamma _n =  f_n(x)\otimes \mu$, $\gamma =  \nu_x\otimes \mu$. For any fixed  $\varphi\in \C(X)$, $\psi \in \C(Y)$, we get
$$ \int _X \varphi(x)\left ( \int _Y \psi(y) d \nu_x \right ) d  \mu = \int _{X\times Y} \varphi(x)\psi(y)\ d\gamma = \lim_{n\rightarrow + \infty}  \int _{X\times Y} \varphi(x)\psi(y)\ d\gamma _n =$$
\begin{equation}\label{weak}= \lim_{n\rightarrow + \infty}\int _X \varphi(x)\left ( \int _Y \psi(y) df_n(x) \right ) d \mu .\end{equation} By density of continuous functions, it follows that 
$ \int _Y \psi(y) d f_n(x)  \rightharpoonup \int _Y \psi(y) d\nu_x$ in Lebesgue spaces of integrable functions.

Since $X $ is quasiconvex,
 %Let $\gamma _n =\mu\otimes \nu_x^n \stackrel{*}{\rightharpoonup}\gamma = \mu \otimes \nu_x$ as in the proof of the above Theorem.
by definition of generalized reformations, it follows that the sequence $f_n$ is equi-Lipschitz on $X $.
%  the following inequality locally holdsfor every  $n\in\N$$$ e_{\gamma _n}(x_0):= \limsup_{x\rightarrow x_0}\frac{W(\nu^n_x,\nu^n_{x_0})}{|x-x_0|}\leq K $$It follows that the sequence $f_n(x)=\nu^n_x$ is locally equi-Lipschitz. 
%Since $X, Y$ are compact, 
By  Ascoli-Arzel\'{a} Theorem, by passing 
to a subsequence we have that $f_n \rightarrow f$ uniformly on compact subsets. 
Since the disintegration is uniquely determined, 
it follows that  $f(x)=\nu _x$ for $\mu$-a.e. $x$. 
Indeed, since the Wasserstein distance metrizes the weak$^*$ convergence of measures ($Y$ is compact), for every $\psi \in \C(Y)$ we have
$$\int_Y \psi df_n(x) \rightarrow \int_Y \psi df(x).$$ Hence, for every $\varphi \in \C(X)$, passing to the limit under the integral sign and by \eqref{weak} we get
$$ \int_X \varphi (x)\left ( \int_Y\psi df(x)\right ) d\mu = $$ $$\lim_{n\rightarrow + \infty} \int_X\varphi (x) \left ( \int_Y \psi df_n(x)\right ) d\mu = 
\int_X \varphi (x)\left ( \int_Y\psi d\nu_x\right ) d\mu.$$
By lemma \ref{convex},  $f_n, f_n^{-1}$ are both locally equi-Lipschitz. % (here by Remark \ref{comp} we get two universal constants $H,K$), it follows that the limit map $f$ is also bi-Lipschitz. Indeed,
%$$|x_1-x_2|=|f_n^{-1}(f_n(x_1))-f_n^{-1}(f_n(x_1))|\leq \frac{1}{H} W(f_n(x_1),f_n(x_2))\leq \frac{K}{H}|x_1-x_2|.$$Passing to the limit as $n\rightarrow + \infty $ we get $$H|x_1-x_2|\leq W(f(x_1),f(x_2))\leq K|x_1-x_2|.$$
It follows that also $f$ is invertible. Indeed, if $y_0=f(x_1)=f(x_2)$, as in the proof of Theorem \ref{min1}, the inverse maps $f_n^{-1}$ are well defined on a small ball $B(y_0, \delta)$.  For a common Lipschitz constant $H$ we compute 
$$ |x_1-x_2|=|f_n^{-1}(f_n(x_1))-f_n^{-1}(f_n(x_2))|\leq H W(f_n(x_1),f_n(x_2))\leq  $$ $$ H\left ( W(f_n(x_1),f(x_1))+W(f(x_1),f(x_2))+ W(f(x_2),f_n(x_2)) \right ).$$ Letting $n\rightarrow +\infty$ we get $x_1=x_2$.
%where $H$ is a common Lipschitz constant for $u_n^{-1}$.  By the arbitrariness of $\eps$ we get the contradiction $x_h^1=x_h^2$.
% locally invertible on an open dense $U_n\subset X$. Since we are dealing with small reformations, by Theorem \ref{invertibility-small} it follows that each $f_n$ is globally invertible %on $U_n$. Since every $f_n$  uniquely extends to $X$,  the maps $f_n$ are homeomorphisms, 
%and locally bi-Lipschitz on $X$ by Lemma \ref{convex}. By Lemma \ref{local-inver}, we have that also $f$ is locally bi-Lipschitz on $X$.
Therefore,  $f\in \GRefo_0 (\mu;\nu)^{H,K}$.
%Moreover, it easily seen that  $e_\gamma (x)=\lip(f)(x)\leq K$.\\
%Let  $f :X\rightarrow \left(\P(Y), W \right)$   such that  $\lip(f)(x_0)=e_\gamma (x_0)$, then $f$ is a locally lipschitz function.
%Let $\gamma_n= f_n\otimes \mu$ be a minimizing sequence for $\mathscr R$. Let $f_n\rightarrow f$ as in the proof of Theorem~\ref{existence-plan}. $\gamma = f\otimes \mu$ has the right marginal requirements. 
Since
$$ 2\leq \int_X \left (e_{\gamma_n}+\frac{1}{e_{\gamma_n}}\right ) d\mu \leq \mathscr R(\gamma_n)\:\:\:\:\forall n\in\N,$$ 
passing to the limit we get
$$\lim_{n\rightarrow +\infty}\int_Xg_n(x) d\mu =2,$$
where $g_n(x)=e_{\gamma_n}+\frac{1}{e_{\gamma_n}}$. 
Passing to a subsequence we have $g_n\rightarrow 2$ a.e.
Since $g_n(x)=\varphi(e_{\gamma_n})$ for $\varphi(t)=t+\frac{1}{t}$, by continuity of $\varphi$ it follows that $e_{\gamma_n}\rightarrow 1$ a.e.
On the other hand, $c_{\gamma_n} \geq \frac{1}{e_{\gamma_n}}$ yielding $\liminf_{n\rightarrow + \infty} c_{\gamma_n}\geq 1$ a.e.
Since  $\gamma_n$ is a minimizing  sequence for $\mathscr R$,
we get
$$ 2=\lim_{n\rightarrow + \infty} \mathscr R(\gamma_n)= 1+ \lim_{n\rightarrow +\infty} \int_X c_{\gamma_n} d\mu.$$
and by Fatou Lemma we infer
$$ 1\leq \int_X \liminf_{n\rightarrow + \infty} c_{\gamma_n}d \mu \leq \lim_{n\rightarrow + \infty} \int_Xc_{\gamma_n} = 1. $$ Therefore, by passing to a subsequence, we also have that $c_{\gamma_n}\rightarrow 1$ a.e.
  Arguing as in the proof of 
Theorem~\ref{iso-plan},  we locally find in $\Omega$  a curve 
$\eta:[0,1]\rightarrow \P(Y) $ such that $e_{\gamma_n}(\eta(t))\rightarrow 1$ a.e. and $l(\eta)\leq |x_1-x_2|+\eps$. Therefore we get
$$W(f_n(x_1),f_n(x_2))\leq \int_0^1e_{\gamma_n}(\eta(t))|\dot \eta |(t) dt.$$
Passing to the limit we obtain
$$W(f(x_1),f(x_2))\leq l(\eta)\leq |x_1-x_2|+\eps.$$ Letting $\eps \rightarrow 0^+$ we obtain the $1$-Lipschitz condition
$$W(f(x_1),f(x_2))\leq|x_1-x_2|.$$
Arguing again as in the proof of 
Theorem~\ref{iso-plan},   we (locally) obtain  $$W(f(x_1),f(x_2))= |x_1-x_2|,$$ hence   $\mathscr R(\gamma)=2$. 
\end{proof}
%~~~~~~~~~~~~~~~~~~~~
\begin{remark}
To recover a global isometry in the above results as in Theorem \ref{main} one should establishes some metric version of Liouville Rigidity Theorems as in Theorem~\ref{rigidity}.
\end{remark}

A natural question concerns the validity of an existence result as in Theorem \ref{min1}. However, we observe that the approach pursued in the proof of such result involve the push-forward of the transport map. Therefore, for generalized reformations, the push-forward of the disintegrations maps is involved. This point of view leads to consider a variational problem over transport classes as introduced in \cite{disintegrations}. The definition of transport classes is the following 
\begin{definition}
Let $\gamma, \eta\in \Pi(\mu,\nu)$ with $\gamma= f(x)\otimes \mu$, $\eta= g(x)\otimes \mu $ be given. We shall say that $\gamma$ and $\eta$ are equivalent (by disintegration), in symbols $\gamma\approx\eta$, if $f_\#\mu= g_\#\mu$.\\ For any $\eta\in \Pi(\mu,\nu)$ with $\eta= g(x)\otimes \mu $,  we shall call transport class any  equivalence class of a transport plan $\eta$  and it  will be denoted by $[\eta]$, i.e.
\begin{equation}
\label{equiv-class}
[\eta]=\{\gamma= f(x)\otimes \mu\;\vert\;f_\#\mu= g_\#\mu\}.
\end{equation}
\end{definition}
For a transport map $u$ the disintegration map is given by $x\mapsto\delta_{u(x)}$. In \cite{disintegrations} it is shown that every such disintegration map leads to the same push-forwarded measure. In other words, all the reformation plans of the form $(I\times u)_\#\mu$ belong to the same transport class. Moreover, the following result holds true
\begin{prop}
Let $u:X\rightarrow Y$ be such that $u_\#\mu=\nu$ and let 
$\eta =(I\times u)_\# \mu= \delta _{u(x)}\otimes \mu$.
If $\gamma\in[\eta]$ then there exists $v:X\rightarrow Y$ such that
$\gamma=\delta _{v(x)}\otimes \mu$, i.e. the transport plan $\gamma$ is concentrated on  the graph of $v$.
\end{prop}
In this perspective, fixed $v:X\rightarrow Y$  such that $v_\#\mu=\nu$,  the variational problem \eqref{p1} studied in Section $3$ could be rephrased as follows \begin{equation}
\min \{ \mathscr{R} (u) \ | \ u\in \Refo(\mu;\nu)^{H,K}\} =\min_{\GRefo(\mu;\nu)^{H,K}} \{ \mathscr {R}(\gamma )\:\vert  \: \gamma \in[ (I\times v)_\# \mu ] \}.
\end{equation}
However, by passing to transport plans, different transport classes arise. 
By the above discussion it seems  natural to fix a transport plan 
$\eta\in\Pi(\mu,\nu)$, 
$\eta = g(x)\otimes \mu$  and %  consider the equivalence class of transport plans $[\eta]$ made by all the generalized reformation plans $\gamma = f(x)\otimes \mu$ such that $f_\# \mu =g_\# \mu$. Hence, for $\eta \in  \GRefo (\mu;\nu)$ it is meaningful 
to consider the variational problem
\begin{equation}
\label{min_generalized}
 \min_{\GRefo(\mu;\nu)^{H,K}} \left \{ \mathscr {R}(\gamma)  \:  \vert\: \gamma \in [\eta ]\right  \}.
\end{equation}

%~~~~~~~~~~~~~~~~~~~~

%%%%%%%%%%%%%%%%%%%%%%%%%%%%%%%%%%%%%%%%%%%%%%%%%%
\begin{theorem}(Existence of optimal reformation  plans)
\label{existence-plan}
Assume \eqref{domain} and $\mu =\L^N\res \Omega$. Let  $\eta \in  \GRefo_0 (\mu;\nu)^{H,K}$ be given. 
Then, for every $p>1$ the variational problem 
\begin{equation}
\label{mingen}\min_{\GRefo_0(\mu;\nu)^{H,K}} \left \{ \mathscr {R}^p(\gamma):=\int_X(c_\gamma^p+e_\gamma^p)d\mu   \:  \vert\: \gamma \in [\eta ]\right  \}\end{equation}

%\eqref{min_generalized} 
admits solutions.
\end{theorem}
\begin{proof}%For $X=\Omega$, 
Let $\gamma _n=f_n(x)\otimes \mu$ be a minimizing sequence. Let $f_n \rightarrow f$ uniformly with respect to the Wasserstein distance as in the proof of Theorem \ref{min2}. 
By Lemma \ref{lsc} we get the lower semicontinuity  of the term $\int _X e_\gamma^p (x) d\mu $. Moreover, by Lemma \ref{reform} %by Theorem~\ref{invertibility-small} 
we get
\begin{equation}
\label{cp}
\int_X c_\gamma^p (x) d\mu = \int _ X \lip^p(f^{-1})(f(x))\ d\mu
\end{equation}
  Since \eqref{domain} $X$  satisfies  the  doubling condition given in Definition~\ref{doubling}   and the  Poincar\'{e} inequality given in Definition~\ref{poincare},  we can  apply 
 the theory of Sobolev spaces  over the subset $f(X)$ of the metric space  $(\P(Y), W, f_\# \mu)$
 (see Appendix B).  Moreover  (see \cite{sob-cheeger}), since for $p>1$ the pointwise Lipschitz constant $\lip(g)$ is the minimal generalized  upper gradient of the locally Lipschitz map $g$ (\cite[Theorem 5.9]{sob-cheeger}) and the Cheeger $p$-energy \eqref{c-energy} is lower semicontinuous  with respect to $L^p$ convergence (\cite[Theorem 2.8]{sob-cheeger}),  by using \eqref{cp} we have 
%~~~~~~~~~~~~~~~~~~~~
\begin{equation}
 \int_X c_\gamma^p (x) d\mu  =  \int _ {\P(Y)} \lip^p(f^{-1})(y)\ d(f_\# \mu )\leq 
  \liminf_{n \rightarrow + \infty} \int _ {\P(Y)}\lip^p(f_n^{-1})(y)\ d( f_\# \mu ).
 \end{equation} 
By taking into account the condition $(f_n)_\# \mu =f_\#\mu$ $\forall n\in \N$, we get
$$\int_X c_\gamma^p (x) d\mu\leq \liminf_{n\rightarrow + \infty}\int _ {\P(Y)}\lip^p(f_n^{-1})(y)\ d(\left (f_n)_\# \mu \right )
= \liminf_{n\rightarrow + \infty} \int _X c_{\gamma _n}^p (x) d\mu .$$
\end{proof}

\subsection{Small reformation plans}
Let $\gamma  \in \GRefo (\mu;\nu)^{H,K}$  %Then there exists an open dense subset $U\subset X$ such that 
%Let $\gamma \in \GRefo (\mu;\nu)^{H,K}$ 
and  $f: X \rightarrow \P(Y)$ be the correspondent disintegration map.
Following the proof of Theorem \ref{invertibility1}, $f$ is locally invertible on an open dense subset $U$ and $N(y,f,U)=D$ is locally constant. In order to prove that actually $N(y,f,U)=1$, fix a small ball $B$ on which $f$ is bi-Lipschitz.  By using   the Metric Area Formula (see \cite{a-k, area-metric, kirc}) we have
 
$$ D\H^N(f(B))=\int_{f(B)}N(y,f,B ) \ d\mathcal H^N(y) =\int_{f^{-1}(f(B))} J(MD(f,x)) \ dx \leq $$ $$\leq \int_{f^{-1}(f(B))} e_f(x)^N\ dx \leq K^N \L^N(f^{-1}(f(B))),$$ 
where $MD(f,x)$ denotes the metric differential introduced in Section $1$, while for any  seminorm $P$ the metric Jacobian is defined by
$$ J(P)= N\omega_N \left ( \int_{S^{N-1}}P(v)^{-N}d \mathcal H^{N-1}(v) \right )^{-1}.$$
For $V=f(B)$ and $\mu = \L^N \res X$ we are led to
$$D\mathcal H^N(V)\leq K^N\L^N(f^{-1}(V))=K^Nf_ \# \mu(V).$$
Therefore,  invertibility for \emph{small} $K$ as in Theorem   \ref{invertibility1} depends on the transport class correspondent to $\Lambda =f_ \# \mu$. Such invertibility property could be obtained for  $\Lambda (V)\leq \mathcal H^N(V)$.  For instance, consider the isometric embedding
$y\mapsto \delta _y$. Let $i(Y)=\Delta\subset \P(Y)$ be the set of Dirac deltas. It follows that $\mathcal H^ N(\Delta)= \mathcal H^ N(i(Y))=\L^N(Y)=1$. Consider $\Lambda$ as the probability measure over $(\P(Y), W)$ defined by
$\Lambda (F)=\int_F \chi_ \Delta (\lambda)\  d \mathcal H ^N(\lambda )$. In such case we have that if $K< \sqrt[N] 2$ then $f$ is globally invertible. Therefore, fixed a transport plan $\eta =f(x)\otimes \mu$ such that $f_ \# \mu =\Lambda$, we get existence of the variational problem of minimizing  $\mathscr {R}^p(\gamma)$ over the set 
$$\{ \gamma \in \GRefo(\mu,\nu)^{H,K} \: : \: \gamma \in [ \eta ] \}, $$
provided of course that such set of reformation plans is not empty. 

%In the general case, let $V$ be a neighborood of $y$ on which  $N(y,f,U)=D$ is  constant. Since $f$ is locally bi-Lipshitz, we can cover    $f^{-1}(V)$  by balls $B(x,r)$ on which $f$ is bi-Lipschitz . Fixed $\eps >0$, by Vitali's covering result we can cover    $f^{-1}(V)$ by a finite number $n$ of such balls $B_i:= B(x_i,r_i)$  up to a measurable set $\L^N(A)<\eps$. Letting $g_i=f_{| {B_i}}$ we get
%$$ \L^N(f^{-1}(V))\leq \sum_{i=1}^n \L^N(B_i)+\eps = \sum_{i=1}^n \L^N(g_i(f(B_i)))+\eps \leq H^N \sum_{i=1}^n \L^N(f(B_i))+\eps \leq $$$$ H^N\mathcalH ^N(\bigcup f(B(x,r)))+\eps =H^N \mathcal H^N (f(f^{-1}(V)))  +\eps \leq H^N \mathcal H^N (V)+\eps .$$Letting $\eps \rightarrow 0^+$ weobtain global invertibility of $f$ under the assumption $HK<\sqrt[N]2$. Therefore, we also get existence of the variational problem of minimizing $\mathscr {R}^p(\gamma)$ over the set $$\{ \gamma \in \GRefo(\mu,\nu)^{H,K} \: : \: HK<\sqrt[N]2 \}, $$provided  that such set of reformation plans is not empty. 

%%%%%%%%%%%%%%%%%%%%%%%%%%%%%%%%%%%%%%%%%%%%%%%%%%%%%%%%%%%%%%%%%%%%%%%%%%%%%%%%%%%%%%%%%%%%%%%%%%%%%%

\appendix
\section{Curves in metric spaces}

For reader convenience here we just  summarize some basic results.
For  analysis in metric
spaces we refer to \cite{a-g-s, a-t, lip, h-lip}. 
For Lipschitz function on a metric space $(X,d)$ we introduce the metric derivative according to the following definition.
\begin{definition}Given a curve $\rho : [a,b] \rightarrow (X,d)$,  the metric derivative at the point $t\in ]a,b[$ is given by
\begin{equation}\label{defmetricd} \lim _{h\rightarrow 0}\frac{d(\rho (t+h), \rho (t))}{h} \end{equation}whenever it exists and in this case we denote it by $|\dot \rho | (t)$.\end{definition} Of course, the above notion of metric derivative coincides with the metric differential  \eqref{m-diff}.
If $\rho : [a,b] \rightarrow (X,d)$ is a Lipschitz curve, by metric  Rademacher Theorem the metric derivative of $\rho$ exists at $\L ^1$-a.e. point in $[a,b]$.
Furthermore, the length of the Lipschitz curve $\rho$  is given by
\begin{equation}
\label{length}
  l(\rho ) = \int _a^b |\dot \rho | (t) dt .
  \end{equation}
We restrict  to the case of $\P(X):=(\P(\Omega),W)$.
The following theorem  relates absolutely continuous  curves 
in $\P(X)$ to the continuity equation. 

\begin{theorem}\label{tcont}
  Let $t\mapsto \rho _t \in  \P(X), t\in [0,1]$,  be a curve. 
If $\rho _t $ is absolutely continuous and $|\dot \rho|\in L^1(0,1)$ is its metric derivative,
then there exists a Borel vector field $v:(t,x) \mapsto v_t(x)$ such
that  
\begin{equation}\label{mince}
v_t \in L^p(X, \rho_t) \ \ \mbox{and}\ \ \|v_t\|_{L^p(X, \rho_t)} \leq
|\dot \rho|(t) \ \ \mbox{for} \ \ \L^1-a.e. \ t \in [0,1] 
\end{equation}
and the continuity equation 
\begin{equation}
\label{eqcont}
 \dot \rho _t%}{\partial t}
+{\rm{div}}( v \rho _t)=0 \ \  in \ \  \ (0,1) \times
X,  \end{equation} where the divergence operator with respect to the spatial variables 
is satisfied in the sense of  distributions. \\
Conversely, if $\rho _t$ satisfies the continuity equation \eqref{eqcont} for some
vector fields $v_t$ such that $\|v_t\|_{L^p(\rho_t)} \in L^1(0,1)$, 
then $ t \mapsto \rho _t$ is  absolutely continuous and 
$$ |\dot \rho|(t) \leq \|v_t\|_{L^p(X, \rho_t)}\ \ \mbox{for} \ \ \L^1-a.e. \
t \in [0,1].$$
\end{theorem}
\begin{remark}
The minimality property (\ref{mince}) uniquely determines a tangent
field $v_t$. We will refer to $v_t$ as the tangent vector associated to the curve $t\mapsto \rho _t$.
The continuity equation (\ref{eqcont}) has been used in the
Monge-Kantorovich theory since its beginning for many applications.
The fact that it characterizes the absolutely continuous curves on the space of
probability measures equipped with the Wasserstein metric was only
recently pointed out and the full proof is contained in \cite{a-g-s}.
\end{remark}
An immediate consequence of the continuity equation is the following
\begin{lemma}\label{der} For every solution $(\rho _t,v_t)$ of the continuity equation (\ref{eqcont}) and for every $f\in \C ^1(X)$ it results
\begin{equation}\label{pdem} \frac{d}{dt} \left( \int \sb X f(x)d\rho _t\right)=
\int \sb X \langle \nabla f(x), v_t(x) \rangle  d\rho _t  \end{equation} in the sense of distributions.\end{lemma}
Actually, it turns out that the map  $f \mapsto \int_X f d \rho _t$ belongs to $W_{loc}^{1,1}(0,1)$. Therefore, formula (\ref{pdem}) holds for a.e. $t\in (0,1)$. We refer the reader to  \cite{a1, a-g-s, granieri} for proofs and more details.  

\section{Sobolev spaces on metric spaces}
\label{sobolev-metric-spaces}
There are several ways to generalize the notion of Sobolev spaces into a metric framework, see for instance \cite{cheeger, point-lip, haj, h-koskela, metric-sobolev,  sob-cheeger, new}. The approach based on the notion of \emph{upper gradient} (see \cite{cheeger, h-koskela,sob-cheeger, new})
 seems to be   more appropriate to the context of this paper.

\begin{definition}\label{ug}  
Let $(X,d_X)$,  $(Y,d_Y)$ be  metric spaces,   let $U\subset X$ be an open subset and  let $u:U\rightarrow Y$ be a given map. 
A Borel function $g:U\rightarrow [0,+\infty]$ is said to be an upper gradient of $u$ if for any unit speed curve $\gamma :[0,l]\rightarrow X$ it results
$$ d_Y(u(\gamma(0)),u(\gamma(l)))\leq \int _0^l g(\gamma (s)) \ ds.$$
\end{definition}
If $u:U\rightarrow Y$ is Lipschitz, then the pointwise Lipschitz constant $\lip (u)$ is an upper gradient for $u$, see \cite{cheeger, point-lip, new}.
For $u\in L^p(U,Y)$, the \emph{Cheeger type $p$-energy} is defined as follows
\begin{equation}
\label{c-energy}
E_p(u)=\inf_{(u_n,g_n)}\liminf_{n \rightarrow + \infty}|g_n|^p_{L^p},
\end{equation}
where the infimum is taken over the sequences $(u_n,g_n)$ such that $g_n$ is an upper gradient of $u_n$ and $u_n\rightarrow u, g_n\rightarrow g $ in $L^p$. By definition \eqref{c-energy} it immediately follows
\begin{equation}\label{p-lsc}
E_p(u)\leq \liminf_{n\rightarrow +\infty }E_p(u_n)
\end{equation} 
whenever $u_n\rightarrow u$ in $L^p$. 
The Cheeger metric $(1,p)$-Sobolev space is defined as
$$ H^{1,p}(U,Y)=\{u\in L^p(U,Y) \ : \ E_p(u)<+\infty \}.$$
We need two more definitions. 
\begin{definition}\label{g-gradient}
A function $g\in L^p$ is called a generalized upper gradient
for $u\in H^{1,p}(U, Y)$ if there exists a sequence $(u_n, g_n )$ such that $g_n$ is an upper
gradient for $u_n$ and $u_n\rightarrow u, g_n\rightarrow g$ in $L^p$.
\end{definition}

%~~~~~~~~~~~~~~~~~~~~

From  Definition~\ref{c-energy}  it follows that $ |g|^p_{L^p}\geq E_p(u)$ whenever $g$ is  a generalized
upper gradient  for $u$.
\begin{definition}\label{m-gradient}
A generalized upper gradient $g$  for a map $u \in H^{1,p}(U, Y)$ is said to be minimal if it satisfies 
$ |g|^p_{L^p}= E_p(u)$ \end{definition}
Under some regularity requirement on the target metric space $Y$, it may be proved (see \cite{sob-cheeger}) that
 every $u \in H^{1,p}(U, Y)$, with $1<p<+\infty$ admits a unique minimal generalized upper gradient $g_u$.
 This minimal generalized upper gradient coincides with the pointwise Lipschitz constant $\lip (u)$ under some geometrical property of the spaces  $(X, \mu), Y$ (see   \cite[Theorem 5.9]{sob-cheeger}). 
 In particular, a crucial role is played by the \emph{doubling condition} and a \emph{weak Poincar\'{e} $(1,p)$-inequality} for the space $(X,\mu)$.

%~~~~~~~~~~~~~~~~~~~~
\begin{definition}
\label{doubling}
A measure $\mu$ over a metric space $X$ is said to be "doubling" if $\mu$ is finite on bounded sets and there exists a constant $C$ such that for every $x\in X$ and every $r>0$ the following inequality holds
\begin{equation}
\label{d-i}
\mu (B(x,2r))\leq C \mu(B(x,r)).
\end{equation}
\end{definition}
%~~~~~~~~~~~~~~~~~~~~
\begin{definition}
\label{poincare}
Let $1\leq p < +\infty$. A metric measure space $(X, d,\mu )$ is said
to satisfy the weak Poincar\'{e} $(1,p)$-inequality  if, for any $s> 0$, there exist constants $C,\Lambda \geq 1$ 
 such that, for any open ball $B(x,r)$ with $0 < r\leq s$, function $f\in L^1(B(x,\Lambda r)))$ and upper gradient $g : B(x,\Lambda r))\rightarrow [0,+\infty ]$ for $f$, the following inequality holds 
 \begin{equation}
 \label{p-i}  
 \mean {B(x,r)}   \left | f- \mean {B(x,r)} f \ d\mu \right |\ d\mu \leq C \left ( \mean {B(x,\Lambda r) }g^p \ d\mu \right ) ^\frac{1}{p}                  
 \end{equation}
\end{definition}
Observe that under some geometrical requirement on $X$, the Poincar\'{e} inequality \eqref{p-i} may be required to hold just for Lipschitz functions $f$ (see \cite{h-lip,h-koskela}). 
The euclidean space $\R^N$ equipped with the Lebesgue measure $\L^N$ is doubling and satisfies the above Poincar\'{e} inequality with $\Lambda =1$. 
Given a square $ Q$    and $\mu =\L^N\res Q$, %for every ball $B(x,r) \subset \R^N $, 
by the inequality 
$$ \frac{1}{2^N}\L^N (B(x,r))\leq  \mu (B(x,r))\leq \L^N (B(x,r)),$$
holding for every ball $ B(x,r)$ of $Q$ and the usual Poincar\'{e} inequality on convex sets,  it follows that $(Q,\mu )$ is doubling and supports the Poincar\'{e} inequality \eqref{p-i}. Since the doubling condition and the Poincar\'{e} inequality are stable under bi-Lipschitz maps, every diffeomorphic (or bi-Lipschitz), with volume preserving maps,  domain $\Omega $ (as balls, see for instance
\cite{fonseca-parry, bilip}) with the same volume of the square $Q$,  equipped with the measure $\nu =\L^N \res \Omega $ is doubling and supports the Poincar\'{e} inequality \eqref{p-i}. 
For more details on the doubling and Poincar\'{e} inequality we refer the reader for instance to \cite{a-t, cheeger, h-koskela, metric-sobolev}.

%~~~~~~~~~~~~~~~~~~~~
\begin{acknowledgment*}
The authors acknowledge  Camillo De Lellis and Luigi Ambrosio for interesting discussions and precious comments. The first author acknowledges the partial support of the INDAM, Istituto Nazionale di Alta Matematica. 
Some of the results of this paper was announced in occasion of the  International Conference: \emph{
Monge-Kantorovich optimal transportation  problem, transport metrics and their applications,  St. Petersburg, Russia,} dedicated to the centenary of L.V. Kantorovich. The first author wish to thanks the audience for discussions and criticism. 
The research of the first author  was also  supported by the 2008 ERC Advanced Grant Project N. 226234  \emph{Analytic Techniques for
Geometric and Functional Inequalities}.

\end{acknowledgment*}

%%%%%%%%%%%%%%%%%%%%%%%%%%%%%%

\end{document}